\documentclass[a4paper,10pt%,draft
]{amsart}
\usepackage[final]{microtype}
\usepackage{amsmath,amssymb,amsthm}
\usepackage[alphabetic, abbrev]{amsrefs}
\usepackage{enumerate}
\usepackage{hyperref}
\usepackage[all]{xy}
\newdir{ >}{{}*!/-5pt/@{>}} %% cf xyguide exercise 14
\SelectTips{cm}{} %% cf xyguide page 9

\numberwithin{equation}{section}
\newtheorem{theorem}{Theorem}[section]
\newtheorem{corollary}[theorem]{Corollary}
\newtheorem{lemma}[theorem]{Lemma}
\newtheorem{proposition}[theorem]{Proposition}
\theoremstyle{definition}
\newtheorem{definition}[theorem]{Definition}
\newtheorem{remark}[theorem]{Remark}
\newtheorem{example}[theorem]{Example}

\newtheorem{construction}[theorem]{Construction}

\newcommand{\arxivlink}[1]{\href{http://arxiv.org/abs/#1}{\texttt{arXiv:#1}}}
\usepackage[pdftex]{graphics,color} 
\usepackage{eso-pic}
\usepackage{scrtime}
\usepackage{ifdraft}
\ifdraft{%
\AddToShipoutPicture{\put(30,30){\resizebox{!}{.4cm}%
   {{{\color[gray]{0.8}\texttt{Preliminary draft, compiled \the\year-\the\month-\the\day~at \thistime}}}}}}
}{}

\newcommand{\fib}{\mathrm{fib}}
\newcommand{\gp}{\mathrm{gp}}
\newcommand{\Der}{\mathrm{Der}}
\newcommand{\Map}{\mathrm{Map}}
\newcommand{\bfHom}{\mathrm{\mathbf{Hom}}}

\newcommand{\sP}{\mathcal{P}}
\newcommand{\sL}{\mathcal{L}}
\newcommand{\IL}{\mathbb{L}}
\newcommand{\IZ}{\mathbb{Z}}
\newcommand{\IN}{\mathbb{N}}
\DeclareMathOperator{\sMod}{sMod}
\DeclareMathOperator{\id}{id}
\DeclareMathOperator{\Mod}{Mod}
\DeclareMathOperator{\Rog}{Rog}
\DeclareMathOperator{\Ext}{Ext}
\DeclareMathOperator{\Ho}{Ho}
\DeclareMathOperator{\str}{str}

\DeclareMathOperator{\dLogAff}{dLogAff}
\DeclareMathOperator{\Hom}{Hom}

\DeclareMathOperator{\fs}{fs}
\DeclareMathOperator{\ab}{ab}
\DeclareMathOperator{\pre}{pre}
\DeclareMathOperator{\Spec}{Spec}

\title{Derived logarithmic geometry I}
\author{Steffen Sagave} \address{Steffen Sagave, Department of Mathematics and Informatics, Bergische Universit{\"a}t Wuppertal, Gau{\ss}\-str. 20, 42119 Wuppertal, Germany} 
\email{sagave@math.uni-wuppertal.de}
\author{Timo Sch\"urg} \address{Timo Sch\"urg, University of Augsburg, 
Universit\"atsstr. 14, 86159 Augsburg, Germany} 
\email{timo.schuerg@math.uni-augsburg.de}
\author{Gabriele Vezzosi} \address{Gabriele Vezzosi, Institut de Math\'ematiques de Jussieu - UMR7586 Batiment Sophie Germain
Case 7012
75205 PARIS Cedex 13}
\email{gabriele.vezzosi@math.imj-prg.fr}

\date{\today}

\begin{document}
\begin{abstract}
  In order to develop the foundations of derived logarithmic geometry,
  we introduce a model category of logarithmic simplicial rings and a
  notion of derived log-\'etale maps and use this to define derived log
  stacks.
\end{abstract}
\maketitle
%\tableofcontents
\section{Introduction}
Before discussing the contents of the present paper, we first want to give some 
motivation why a solid theory of derived logarithmic geometry is a 
desirable thing to have.

An important application of logarithmic geometry has been to control 
degenerations. A typical example is given by a dominant morphism $f 
\colon X \to C$ from a smooth scheme $X$ to a pointed curve $(C,p)$, 
where we assume that the restriction $X \setminus X_p \to C \setminus p$ 
is smooth and the fibre $X_p$ is a normal crossing divisor. If we denote 
by $j \colon X \setminus X_p \to X$ and $i \colon X_p \to X$ the 
inclusions, then $i^*(j_* \mathcal{O}^{\times}_{X \setminus X_p})\to 
\mathcal{O}_{X_p}$ defines a log structure on $X_p$. In the opposite 
vein, given a normal crossing variety~$Y$, the existence of certain 
logarithmic structures on $Y$ helps in determining if $Y$ can be 
obtained as the fibre of morphism $X \to C$ as above (see 
\cite{abramovich}*{Section 5} and the references therein for this point 
of view).

A further striking example where logarithmic geometry helps to control 
degenerations is given by the Deligne--Mumford compactification of the 
moduli space of curves. This compactification can also be obtained by 
studying the moduli problem of stable log-smooth curves satisfying a 
certain basicness condition. Since logarithmic geometry incorporates 
degenerations, the moduli space of log-smooth curves is immediately 
compact. An overview over these topics can be found in~\cite{abramovich}.

On the other hand, derived algebraic geometry has been successfully 
applied to study hidden smoothness in moduli spaces. A typical example 
is given by the moduli space of morphisms between a smooth curve and a 
smooth projective variety. Even for smooth domain and target, this 
moduli space can be horribly singular and much larger than the expected 
dimension. Studying the same moduli problem in derived algebraic 
geometry leads to an interesting ``nilpotent'' structure on the moduli 
space \cite{hagII}*{Corollary 2.2.6.14}. This structure provides the 
algebraic-geometric counterpart to deforming to transversal 
intersection. Equipped with this nilpotent structure, the moduli space 
becomes quasi-smooth, the immediate generalization to derived algebraic 
geometry of local complete intersection. The quasi-smooth structure 
induces a 1--perfect obstruction theory  and a virtual fundamental class 
in the expected dimension on the underlying moduli space, which is the 
key to many enumerative invariants.

Logarithmic and derived algebraic geometry naturally meet in the study
of degenerations of moduli spaces. Suppose we are given a morphism $f
\colon X \to S$ as above. We would then like to understand how some
moduli space attached to a smooth fiber interacts with the
corresponding moduli space of the fiber $X_s$.  If the moduli spaces
are quasi-smooth, one would ideally want to compute enumerative
invariants of the smooth fiber in terms of enumerative invariants of
the components of $X_s$.

In case $X_s$ only consists of two components, this has been indeed 
carried out by Jun Li in \cites{li1, li2}. Instead of using log 
geometry, Li constructs an explicit degeneration of the moduli space of 
stable maps. The most difficult part in Li's theory is to find a perfect 
obstruction theory on the moduli space attached to the fiber $X_s$.  
Using this degeneration, he is able to prove a formula for 
enumerative invariants that since has found many applications.

Gross and Siebert \cite{siebert} have recently observed that one can
circumvent these difficulties by working in the category of
logarithmic schemes.  The moduli space attached to the fiber $X_s$
should just be the corresponding moduli functor taken in the category
of logarithmic schemes, where $X_s$ is equipped with its natural
logarithmic structure.  If on top of this we want the moduli space
attached to $X_s$ to carry a 1-perfect obstruction theory, one is
naturally led to consider derived logarithmic geometry. The correct
functor that combines both the degeneration aspects as well as hidden
smoothness is a moduli functor living in the category of derived
logarithmic schemes or stacks.

Besides applications to degenerations of quasi-smooth moduli spaces,
there are also other areas where such a theory might be useful. Much
of the work on logarithmic geometry has been concerned with $p$-adic
and arithmetic aspects. Recently, Beilinson \cite{beilinson} has used
derived logarithmic geometry to prove a $p$-adic
Poincar\'e lemma. Much more material on this can be found in
\cite{bh}. It may also be interesting to extend the framework of
derived log geometry developed here to the homotopy theoretic notion
of \emph{logarithmic ring spectra} developed by Rognes
in~\cite{rog-log} in order to study moduli problems for structured
ring spectra.

We hope that now the reader is convinced that it would be desirable to
have a solid theory of derived logarithmic geometry.  The aim of this
work is to begin providing such foundations. The essential starting
point for derived algebraic geometry is that the category of
simplicial rings forms a well-behaved model category. In Sections 1 to
3 we provide a model category $s\mathcal L$ of \emph{logarithmic
  simplicial rings}. Its objects are simplicial objects in the
category of pre-log rings, and the fibrant objects in this model
structure satisfy a log condition analogous to that of a log ring.
Besides that, we give a model category description of the group
completion of simplicial commutative monoids and outline how this
leads to a notion of \emph{repletion} for augmented simplicial
commutative monoids. Although the repletion is not necessary for setting
up the model category $s\mathcal L$, it might become relevant for a
further development of the theory. All the model structures developed
in this part have counterparts in the context of structured ring
spectra that complement Rognes' work on \emph{topological logarithmic
  structures}~\cite{rog-log}.

In Sections 4 to 6 we develop the theory of \'etale and smooth
morphisms between logarithmic simplicial rings. The key ingredient in
defining these notions is the logarithmic cotangent complex. We define 
the logarithmic cotangent complex as the complex that represents the 
derived functor of logarithmic derivations. Since for a logarithmic ring 
$(A,M)$ the category of $A$-modules is equivalent to the category of 
abelian objects in the category of strict logarithmic rings over 
$(A,M)$, this exhibits the logarithmic cotangent complex as the left 
derived functor of abelianization, which is very close to Quillen's 
original definition for ordinary rings.  This coincides with Gabber's 
definition in \cite{olsson}*{\S 8}, and we prove that it also 
corresponds to Rognes' definition for structured ring spectra 
in~\cite{rog-log}. We also compare our notions of log-smooth and 
log-\'etale maps to the
definitions given by Kato in terms of lifting properties with respect
to strict square zero extensions.

In Section 7 we glue logarithmic simplicial rings to form derived 
logarithmic schemes and derived logarithmic $n$--stacks. We conclude 
this section with some speculations about the correct notion of
log-modules. In Section 8 we explain how to set up a derived version
of the logarithmic moduli of stable maps introduced by Gross and
Siebert.

\subsection*{Acknowledgments} Bhargav Bhatt has also recently started 
laying the foundations of derived log geometry in~\cite{bh}*{\S 4}.  The 
third author wishes to thank him for interesting exchanges on the 
subject.  Thanks are also due to Vittoria Bussi for her interest and 
useful discussions. The authors thank the referee for a quick and 
careful reading and a large number of suggestions improving the text.

\subsection*{Notations} If $k$ is a base commutative ring, then the 
category of pre-log $k$-algebras will consist of triples $(A, M, 
\alpha\colon M \to (A,\cdot))$ where $A$ is a commutative $k$-algebra, 
$M$ is a commutative monoid and $\alpha$ is a morphism of commutative 
monoids, and the morphisms $(A, M, \alpha\colon M \to (A,\cdot)) \to (B, 
N, \alpha\colon N \to (B,\cdot))$ will be pairs $(f\colon A \to 
B,f^{\flat}:M \to N)$, where $f$ is a map of $k$-algebras and 
$f^{\flat}$ a map of commutative monoids, commuting with the structure 
maps.  When the base ring is $k=\mathbb{Z}$ we will simply speak about 
pre-log rings.

\section{Simplicial commutative monoids}\label{sec:simp-com-mono}
In the following, we let $\mathcal M$ be the category of commutative
monoids, $\mathcal{AB}$ be the category of abelian groups, and
$\mathcal R$ be the category of commutative rings. Moreover, $\mathcal S$
denotes the category of simplicial sets, and $s\mathcal M$,
$s\mathcal{AB}$, and $s\mathcal R$ denote the categories of simplicial
objects in commutative monoids, abelian groups, and commutative rings.

The categories $s\mathcal M$, $s\mathcal{AB}$, and $s\mathcal R$ are
simplicial categories (as for example defined
in~\cite{Goerss-J_simplicial}*{II.Definition 2.1}). This means that
they are enriched, tensored, and cotensored over the category of
simplicial sets. In each case, the tensor $X \otimes K$ of an object
is the realization of the bisimplicial object $[n] \mapsto
\coprod_{K_n}X$ where $\coprod$ is the coproduct in the respective
category. The simplicial mapping spaces are given by
$\bfHom(X,Y)_n = \Hom(X\otimes \Delta^n,Y)$, and the
cotensor is defined on the underlying simplicial sets. There exist
well-known model structures on these categories:

\begin{proposition}\label{prop:model-str-on-sR-and-sM}
  The categories of simplicial commutative rings $s\mathcal R$,
  simplicial abelian groups $s\mathcal{AB}$, and simplicial
  commutative monoids $s\mathcal M$ admit proper simplicial 
  cellular model structures. In all three cases, a map is a fibration
  (resp. weak equivalence) if and only if the underlying map of simplicial sets is
  a fibration (resp. weak equivalence).
\end{proposition}
We refer to these model structures as the \emph{standard} model structures
on these categories.
\begin{proof} The existence of these model structures is provided
  by~\cite[II.4 Theorem 4]{Quillen_homotopical}
  or~\cite[II.Corollary 5.6]{Goerss-J_simplicial}. Right properness
  is inherited from simplicial sets.  Since the cartesian product is
  the coproduct of commutative monoids, left properness of $s\mathcal
  M$ is a consequence of~\cite[Theorem 9.1]{Rezk_every-proper}. Left
  properness of $s\mathcal{AB}$ may for example be
  established using the Dold-Kan correspondence. For $s\mathcal R$,
  left properness is verified
  in~\cite[Lemma 3.1.2]{Schwede_spectra-cotangent}.  

  Applying the respective free functors from simplicial sets to the
  usual generating cofibrations and generating acyclic cofibrations
  for $\mathcal S$ shows that all three categories are cofibrantly
  generated. The argument given
  in~\cite[Appendix A]{Sagave-S_group-compl} can be adopted to show
  that $s\mathcal M$ and $s\mathcal R$ are cellular. Cellularity of
  $s\mathcal{AB}$ can be checked from the definition.
\end{proof}

\subsection{Group completion} For the rest of this section we focus on
the category of simplicial commutative monoids. This category is
pointed by the constant simplicial object on the one point
monoid. Hence $s\mathcal M$ is a pointed simplicial model category,
i.e., it is tensored, cotensored and enriched over the category of
pointed simplicial sets. The tensor with the pointed simplicial set
$S^1 = \Delta^1 / \partial \Delta^1$ is isomorphic to the bar
construction on a simplicial commutative monoid. It follows that the
functors $B(M) = M \otimes S^1$ and $\Omega(M)= M^{S^1}$ form a
Quillen adjunction $B \colon s\mathcal M \rightleftarrows s\mathcal M
\colon \Omega$ with respect to the standard model structure.

\begin{definition}
  A simplicial commutative monoid $M$ is \emph{grouplike} if the 
  commutative monoid $\pi_0(M)$ is a group.
\end{definition}

Forming the adjoint of the fibrant replacement $BM \to
(BM)^{\fib}$ of $BM$ in $s\mathcal M$ provides a natural
transformation
\begin{equation}\label{eq:Omega-B-functor}
\eta_M \colon M \to  \Omega((BM)^{\fib}). 
\end{equation}
It is immediate that $\Omega((BM)^{\fib})$ is always
grouplike. The map $\eta_M$ is known as the group completion of $M$.
Below we will compare it with two other ways of forming a group
completion.
\begin{lemma}\label{lem:grp-compl-equivalence-on-gplike}
If $M$ is grouplike, then $M \to  \Omega((BM)^{\fib})$ is a weak equivalence.
\end{lemma}
\begin{proof}
  We may assume that $M$ is fibrant. Writing $E_{\bullet}M =
  B_{\bullet}(*,M,M)$ for the bisimplicial set whose realization is
  the simplicial set $EM$, an application of the Bousfield--Friedlander
  theorem~\cite[Theorem B.4]{Bousfield-F_Gamma-bisimplicial} shows that the
  realization of the degree-wise pullback square
\[\xymatrix@-1pc{ \mathrm{const}\, M \ar[r] \ar[d] & E_{\bullet} M \ar[d] \\
  {*} \ar[r] & B_{\bullet} M }\] provides a homotopy fiber sequence $M
\to EM \to BM$. Since $EM$ is contractible, it follows that $M \to
\Omega((BM)^{\fib})$ is a weak equivalence.
\end{proof}
We now let $C$ be the free simplicial commutative monoid on a point,
i.e., the simplicial commutative monoid obtained applying the free
commutative monoid functor on sets degree-wise to $\Delta^0$. Then we apply $\Omega((B(-))^{\fib})$
to form the group completion of $C$ and choose a factorization
\begin{equation}\label{eq:definition-xi}
  \xymatrix{C \ar@{>->}[r]^{\xi} & C' \ar@{->>}[r]^-{\sim} &   
    \Omega((BC)^{\fib})}
\end{equation}
of $\eta_C$ into a cofibration $\xi$ followed
by an acyclic fibration.
\begin{lemma}\label{lem:bar-con-group-compl}
The map $B\xi \colon BC \to BC'$ is a weak equivalence.
\end{lemma}
\begin{proof}
  Since $BC$ and $BC'$ are connected as simplicial sets, it is enough
  to show that $\Omega((B\xi)^{\fib})$ is a weak
  equivalence. By construction of $\xi$, this reduces to showing that
  $\Omega((B\eta_C)^{\fib})$ is a weak equivalence. The
  composite of $B\eta_C$ with the
  adjunction counit $\varepsilon_D \colon B\Omega D \to D$ on $D =
  (BC)^{\fib}$ is the fibrant replacement of $BC$. Hence it is
  enough to show that $\varepsilon_D$ becomes a weak equivalence after
  applying $\Omega((-)^{\fib})$. The composite of
  $\Omega((\varepsilon_D)^{\fib})$ with the group completion
  map $\eta_{\Omega D}$ is the weak equivalence $\Omega(D \to
  D^{\fib})$. Hence it is enough to see that $\eta_{\Omega D}$
  is the weak equivalence, and this follows from the last lemma.
\end{proof}
The next lemma shows that we may view $\xi\colon C \to C'$ as the
group completion in the universal example. To phrase it, recall that
an object $X$ in a simplicial model category $\mathcal C$ is local
with respect to a cofibration $U \to V$ in $\mathcal C$ if $X$ is
fibrant and the induced map of simplicial sets $\bfHom(V,X) \to
\bfHom(U,X)$ is an acyclic fibration.
\begin{lemma}\label{lem:char-local}
An object in $s\mathcal M$ is $\xi$-local if and only if it is
fibrant and grouplike. 
\end{lemma}
\begin{proof}
Let $M$ be $\xi$-local. Then 
\[ s\mathcal M(C',M) \cong \bfHom(C',M)_0 \to
\bfHom(C,M)_0\cong s\mathcal M(C,M)\] is surjective. Hence every
map $C \to M$ extends over $C'$. Passing to connected components, this
means that any homomorphism $(\mathbb N,+)\cong \pi_0(C) \to \pi_0(M)$ extends
over the group completion $(\mathbb N,+) \to (\mathbb Z,+)$. This
implies that $M$ is grouplike.

Now let $M$ be fibrant as a simplicial set and grouplike. Since $M \to
\Omega((BM)^{\fib})$ is a weak equivalence by
Lemma~\ref{lem:grp-compl-equivalence-on-gplike}, it is enough to show
that \[\bfHom(C',\Omega((BM)^{\fib})) \to
\bfHom(C,\Omega((BM)^{\fib}))\] is a weak equivalence. By adjunction, this
map is isomorphic to 
 \[\bfHom(BC',(BM)^{\fib}) \to
\bfHom(BC,(BM)^{\fib}),\]
and the claim follows from Lemma~\ref{lem:bar-con-group-compl} and the 
fact that $s\mathcal M$ is simplicial.
\end{proof}
The previous lemma enables us to view the group completion of
simplicial commutative monoids as a fibrant replacement in an
appropriate model structure:
\begin{proposition}\label{prop:group-completion-model-str}
  The category of  simplicial commutative monoids $s\mathcal M$ admits
  a left proper simplicial cellular \emph{group
    completion model structure}. 

  The cofibrations in this model structure are the same as in the
  standard model structure. A map $M \to N$ is a weak equivalence if
  and only if the induced map $BM \to BN$ is a weak equivalence of simplicial
  sets.  An object is fibrant if and only if it is both fibrant as a
  simplicial set and grouplike.

  The fibrant replacement $\xymatrix@1{M \; \ar@{>->}[r]&\; M^{\gp}}$ in the group
  completion model structure is weakly equivalent to $\eta_M$.
\end{proposition}

\begin{proof}%[Proof of Proposition~\ref{prop:group-completion-model-str}]
  The desired model structure is defined as the left Bousfield
  localization of the standard model structure with respect to the
  single map~$\xi$. The existence of this model structure, the
  characterization of the cofibrations, and the fact that it is left
  proper, simplicial, and cellular follow
  from~\cite[Theorem
    4.1.1]{Hirschhorn_model}. Lemma~\ref{lem:char-local} provides the description of the
  fibrant objects.

  Now let $M \to M^{\gp}$ be a fibrant replacement in the group
  completion model structure and consider the square
  \[\xymatrix@-1pc{
    M \ar[d] \ar[r] & \Omega((B(M))^{\fib})\ar[d] \\
    M^{\gp}\ar[r] &
    \Omega((B(M^{\gp}))^{\fib}). }\] The bottom
  horizontal map is a weak equivalence by
  Lemma~\ref{lem:grp-compl-equivalence-on-gplike}. Using the universal
  property of the left Bousfield
  localization~\cite[Proposition 3.3.18]{Hirschhorn_model} and
  Lemma~\ref{lem:bar-con-group-compl}, it follows that $B(M) \to
  B(M^{\gp})$ is a weak equivalence. Hence the right vertical
  map is a weak equivalence. This provides the desired
  characterization of the fibrant replacement.

  By the previous argument and~\cite[Theorem
  3.2.18]{Hirschhorn_model}, a map $M \to N$ is a weak equivalence in
  the group completion model structure if and only if it becomes a
  weak equivalence when applying $\Omega((B(-))^{\fib})$. This is the
  case if and only if $B(M) \to B(N)$ is a weak equivalence.
\end{proof}
\begin{remark}
  Simplicial commutative monoids have a topological analog known as
  \emph{commutative $\mathcal I$-space monoids}. These provide
  strictly commutative models for the more common \emph{$E_{\infty}$
    spaces}. The group completion model structure of the previous
  proposition is analogous to the group completion model structure for
  commutative $\mathcal I$-space monoids developed
  in~\cite[Theorem 1.3]{Sagave-S_group-compl}. The proofs of
  Proposition~\ref{prop:model-str-on-sR-and-sM} and of
  Lemmas~\ref{lem:grp-compl-equivalence-on-gplike}
  and~\ref{lem:char-local} closely follow the corresponding statements
  in~\cite[\S 5]{Sagave-S_group-compl}.
\end{remark}
The fact that the fibrant replacement $M \to M^{\gp}$ is always a 
cofibration easily provides the following universal property: 
\begin{corollary}
Every map $M \to N$ of simplicial commutative monoids with $N$ fibrant
and grouplike extends over the group completion $\xymatrix@1{M \; \ar@{>->}[r]&\; M^{\gp}}$. \qed
\end{corollary}\begin{remark}
The example discussed in~\cite[\S 5.7]{Bousfield-F_Gamma-bisimplicial} shows 
that the group completion model structure is not right proper.
\end{remark}
A different way of group completing  a simplicial
commutative monoid is to apply the usual group completion of
commutative monoids degree-wise. As a functor 
\begin{equation}
 (-)^{\textrm{deg-gp}}\colon s\mathcal M \to s\mathcal{AB},  
\end{equation}
this construction is left adjoint to the forgetful functor. The
resulting natural transformation $M \to M^{\textrm{deg-gp}}$ is indeed
equivalent to the group completion considered above:
\begin{lemma}\label{lem:deg-gp-is-fib-rep}
The map $M \to M^{\textrm{deg-gp}}$ is a weak equivalence
with fibrant codomain with respect to the group completion model structure. 
\end{lemma}
\begin{proof}
  Since simplicial abelian groups are fibrant as simplicial sets, it
  follows that $M^{\textrm{deg-gp}}$ is fibrant in the group
  completion model structure. Quillen's analysis of
  $M^{\textrm{deg-gp}}$
  in~\cite[Propositions Q1 and Q2]{Friedlander-M_filtrations} implies
  that $BM \to B(M^{\textrm{deg-gp}})$ is a weak equivalence of
  simplicial sets.
\end{proof}
\begin{corollary}
The degree-wise group completion is the left adjoint in a Quillen equivalence
between the category of simplicial commutative monoids with the group completion 
model structure and the category of simplicial abelian groups. 
\end{corollary}
\begin{proof}
  Since all objects in $s\mathcal{AB}$ are fibrant as simplicial sets
  and grouplike, the forgetful functor $U\colon s\mathcal{AB}\to
  s\mathcal M$  preserves fibrant objects. Hence it follows
  from~\cite[Proposition 3.3.16]{Hirschhorn_model} that $U$ preserves
  fibrations. Since weak equivalences between fibrant objects in the
  group completion model structure are precisely the underlying weak
  equivalences of simplicial sets, it follows that a map $f$ in
  $s\mathcal{AB}$ is a weak equivalence if $U(f)$ is. Hence $U$ is a
  right Quillen functor. Together with the previous lemma, this
  implies that $((-)^{\textrm{deg-gp}}, U)$ is a Quillen equivalence.
\end{proof}
\begin{corollary}
  The homotopy category of simplicial abelian groups is equivalent to
  the homotopy category of grouplike simplicial commutative
  monoids.\qed
\end{corollary}

\begin{remark}
  We have seen that the derived adjunction unit $M \to
  \Omega((BM)^{\fib})$, the fibrant replacement $M \to
  M^{\gp}$ in the group completion model structure, and the
  degree-wise group completion $M \to M^{\textrm{deg-gp}}$ provide
  three equivalent ways of forming group completions of simplicial
  commutative monoids.
\end{remark}
The following result will be used in Section~\ref{sec:log-der}:
\begin{corollary}\label{cor:hocart-sAb-sM}
A commutative square  \begin{equation}\label{eq:hocart-sAb-sM}\xymatrix@-1pc{M \ar[r] \ar[d] & P \ar[d] \\ N \ar[r] & Q}
\end{equation} of simplicial abelian groups is homotopy cocartesian
in $s\mathcal{AB}$ if and only if it is homotopy cocartesian when viewed
as a square in $s\mathcal M$.
\end{corollary}
\begin{proof}
  By Lemma~\ref{lem:deg-gp-is-fib-rep}, the left Quillen functor
  $(-)^{\textrm{deg-gp}}\colon s\mathcal M \to s\mathcal{AB}$ sends
  weak equivalences between not necessarily cofibrant objects to weak
  equivalences.  So if~\eqref{eq:hocart-sAb-sM} is homotopy
  cocartesian in $s\mathcal M$, applying $(-)^{\textrm{deg-gp}}$ shows
  that the square is homotopy cocartesian in
  $s\mathcal{AB}$. Let~\eqref{eq:hocart-sAb-sM} be homotopy
  cocartesian in $s\mathcal{AB}$. We choose a factorization
  $\xymatrix@1{M \,\ar@{>->}[r] & N^c \ar[r]^{\sim} & N}$ in
  $s\mathcal M$ and deduce that $N^c\coprod_M P \to Q$ is a weak
  equivalence in $s\mathcal M$ since it is a weak equivalence after
  applying $(-)^{\textrm{deg-gp}}$ and its domain and codomain are
  grouplike.
\end{proof}
\subsection{Repletion}\label{subsec:repletion}
Many of the conditions on commutative monoids that are useful in
logarithmic geometry do not appear to provide homotopy invariant
notions when imposing them in each level of a simplicial commutative
monoids. As explained by Rognes in~\cite[Remark 3.2]{rog-log}, the
notion of \emph{repletion} for  commutative
monoids~\cite[\S 3]{rog-log} and for commutative $\mathcal I$-space
monoids~\cite[\S 8]{rog-log} is made to overcome this difficulty
in one relevant instance. Repletion has already proved useful for
the definition of logarithmic topological Hochschild homology
in~\cite[\S 8]{rog-log} and~\cite{Rognes-S-S_log-THH}.  The close
relation between repletion and a group completion model structure on
commutative $\mathcal I$-space monoids explained
in~\cite[\S 5.10]{Sagave-S_group-compl} makes it easy to adopt this
to simplicial commutative rings.
\begin{definition}
Let $M \to N$ be a map of simplicial commutative  monoids and let
\[ \xymatrix{M \ar@{>->}[r]^{\sim}  & M^{\mathrm{rep}} \ar@{->>}[r] & M}\]
be a factorization of this map in the group completion model structure. 
Then $M \to M^{\mathrm{rep}}$ is the \emph{repletion} of $M$ over $N$. 
\end{definition}
General properties of left Bousfield localizations imply that as an
object under $M$ and over $N$, the repletion $M^{\mathrm{rep}}$ is
well defined up to weak equivalence in the standard model structure
on $s\mathcal M$.
\begin{definition}\label{def:virt-surj-and-exact}
A  map of simplicial commutative  monoids $M \to N$ is \emph{virtually surjective}
 if the induced homomorphism $(\pi_0(M))^{\gp} \to (\pi_0(N))^{\gp}$ is surjective (compare~\cite[Definition 3.6]{rog-log}).
It is \emph{exact} \cite{ka} if the following square is homotopy 
cartesian in the standard
model structure on $s\mathcal M$:
\[\xymatrix@-1pc{
M \ar[r] \ar[d] & M^{\gp} \ar[d] \\
N \ar[r] & N^{\gp}}
\]
\end{definition}
The next proposition states that for a virtually surjective map,
repletion enforces exactness and can be defined by only using group completions.
\begin{proposition}\label{prop:exact-replete} Let $M \to N$ be a virtually surjective map of simplicial commutative  monoids. 
\begin{enumerate}[(i)]
\item The canonical map $M^{\gp} \to (M^{\mathrm{rep}})^{\gp}$ is a weak equivalence
in the standard model structure. 
\item The repletion $M^{\mathrm{rep}}$ is weakly equivalent  to the map from $M$ into the homotopy
  pullback of $N \to N^{\gp} \leftarrow M^{\mathrm{rep}}$ (with respect to the standard model
structure). 
\item The map $M^{\mathrm{rep}} \to N$ is exact. 
\end{enumerate}
\end{proposition}
\begin{proof}
  The properties of the group completion model structure imply
  (i). For (ii), one can use a similar argument as in the proof
  of~\cite[Proposition 5.16]{Sagave-S_group-compl}. The key
  ingredient is the Bousfield--Friedlander
  theorem~\cite[Theorem B.4]{Bousfield-F_Gamma-bisimplicial} that
  compensates for the missing right properness of the group completion
  model structure. Part (iii) follows from (i) and~(ii).
\end{proof}
If $M$ is a simplicial commutative monoid under and over $N$, then $M
\to N$ is automatically virtually surjective, and passing to the
repletion ensures exactness of the augmentation. 

\section{Logarithmic simplicial rings}
The functor sending a commutative ring $A$ to its underlying multiplicative mon\-oid $(A,\cdot)$ is
right adjoint to the integral monoid ring functor $\mathbb Z[-]$ from commutative monoids
to commutative rings. Applying this adjunction degree-wise provides an adjunction 
\begin{equation}\label{eq:sM-sR-adjunction}
\mathbb Z[-]\colon s \mathcal M \rightleftarrows s \mathcal R \colon (-,\cdot)
\end{equation}
between the associated categories of simplicial objects. The following definition is 
 the obvious generalization of the pre-log structures introduced by Kato in~\cite{ka}.
\begin{definition}
  A \emph{pre-log structure} $(M,\alpha)$ on a simplicial commutative ring $R$ is a
  simplicial commutative monoid $M$ together with a map of simplicial
  commutative monoids $\alpha\colon M \to (A,\cdot)$. A simplicial
  commutative ring $R$ together with a pre-log structure $(M,\alpha)$
  is called a \emph{pre-log} simplicial ring. It is denoted by $(A,M,
  \alpha)$ or simply by $(A,M)$ if $\alpha$ is understood from the
  context.

  A map of simplicial pre-log rings $(A,M) \to (B,N)$ is a pair $(f,f^{\flat})$ of
  maps $f \colon A \to B$ in $s\mathcal R$ and $f^{\flat} \colon M \to
  N$ in $s\mathcal M$ such that the obvious square commutes. We write $s\mathcal P$
for the resulting category of simplicial commutative pre-log rings. 
\end{definition}
Viewing pre-log simplicial rings as simplicial objects in pre-log
rings, the same arguments as in the case of $s\mathcal M$ and
$s\mathcal R$ show that $s\mathcal P$ is a simplicial category. Since
$\mathbb Z[-]$ preserves coproducts, it is immediate that
$(A,M)\otimes K \cong (A \otimes K,M\otimes K)$.  Expressing the
compatibility of the two components of a map of pre-log simplicial
rings as a pullback, it follows that the mapping spaces in $s\mathcal
P$ are related to the mapping spaces in $s\mathcal M$ and $s\mathcal
R$ by a pullback square
\[\xymatrix@-1pc{
\bfHom_{s\mathcal P}((A,M),(B,N)) \ar[r] \ar[d] & \bfHom_{s\mathcal R}(A,B)\ar[d]\\
\bfHom_{s\mathcal M}(M,N)\ar[r]& \bfHom_{s\mathcal M}(M,(B,\cdot))}
\]
\subsection{The pre-log model structures}
Since the adjunction~\eqref{eq:sM-sR-adjunction} is a Quillen
adjunction with respect to the model structures from
Proposition~\ref{prop:model-str-on-sR-and-sM}, we obtain the two model
structures on $s\mathcal P$ described in the next two propositions:
\begin{proposition}\label{prop:inj-pre-log-model-str}
  The category of simplicial pre-log rings $s\mathcal P$ admits an
  \emph{injective} proper simplicial cellular model structure where
  $(f,f^{\flat}) \colon (A,M) \to (B,N)$ is
\begin{itemize}
\item a weak equivalence (or a cofibration) if both $f$ and
  $f^{\flat}$ are weak equivalences (or cofibrations) in the standard
  model structures on $s\mathcal R$ and $s\mathcal M$ and
\item a fibration if $f$ is a fibration in $s\mathcal R$
  and the induced map $M \to (A,\cdot)\times_{(B,\cdot)}N$ is a
  fibration in $s\mathcal M$.
\end{itemize}
\end{proposition}
We call this model structure the \emph{injective
pre-log model structure} and write $s\mathcal P^{\mathrm{inj}}$ for the resulting model
category. The fibrant objects are called
\emph{pre-fibrant}.
\begin{proof}
The existence of this model structure is established by standard lifting arguments.
Using Lemma~\ref{lem:comparing-lifting-conditions} below, one can check that the
generating cofibrations $I_{s\mathcal M}$ and $I_{s\mathcal R}$ for $s\mathcal M$ and $s\mathcal R$
give rise to a set
\[ \{(\mathbb Z[L],K)\to(\mathbb Z[L],L) \,|\, K\to L \in I_{s\mathcal M}\} \cup \{(i,id_{*}) \,|\, i \in I_{s\mathcal R} \}\]
of generating cofibrations for $s\mathcal P$, and similarly for the generating acyclic cofibrations. 
\end{proof}
Similarly, we get a \emph{projective pre-log model structure} denoted by $s\mathcal P^{\mathrm{proj}}$: 
\begin{proposition}\label{prop:proj-pre-log-model-str}
  The category of simplicial pre-log rings $s\mathcal P$ admits a
  \emph{projective} proper simplicial cellular model structure where
  $(f,f^{\flat}) \colon (A,M) \to (B,N)$ is
\begin{itemize}
\item a weak equivalence (or a fibration) if both $f$ and
  $f^{\flat}$ are weak equivalences (or fibrations) in the standard
  model structures on $s\mathcal R$ and $s\mathcal M$ and
\item a cofibration if $f^{\flat}$ is a cofibration in $s\mathcal M$
  and the induced map $\mathbb{Z}[N]\otimes_{\mathbb{Z}[M]}A \to B$ is a
  cofibration in $s\mathcal R$.
\end{itemize}
\end{proposition}
\begin{proof}
  Again this follows by standard lifting arguments.  In this case the
  generating cofibrations $I_{s\mathcal M}$ and $I_{s\mathcal R}$ for
  $s\mathcal M$ and $s\mathcal R$ give rise to a set
\[ \{(\mathbb Z[i],i) \,|\, i \in I_{s\mathcal M}\} \cup \{(i,id_{*}) \,|\, i \in I_{s\mathcal R} \}\]
of generating cofibrations for $s\mathcal P$, and similarly for the generating acyclic cofibrations. 
\end{proof}
\begin{corollary}
  The identity functor from simplicial pre-log rings with the
  projective model structure to simplicial pre-log rings with the
  injective model structure is the left Quillen functor of a Quillen
  equivalence.\qed
\end{corollary}
\begin{remark}The corollary implies that the two model structures are
  equivalent for many purposes. However, as we will see in
  Section~\ref{subsec:log-model} below, the fact that the injectively
  fibrant objects $(A,M)$ have the property that the structure map $M
  \to (A,\cdot)$ is a fibration makes the injective model structure
  more convenient for the purpose of \emph{log} structures.
\end{remark}

If $(A,M,\alpha)$ is a pre-log simplicial ring, we write
$(A,\cdot)^{\times}$ for the sub simplicial commutative monoid of
invertible path components $(A,\cdot)^{\times} \subset
(A,\cdot)$, i.e, the sub simplicial commutative monoid of $(A,\cdot)$ consisting of those simplices whose vertices represent units in the multiplicative
monoid $\pi_0(A)$. Using $(A,\cdot)^{\times}$, we form the following pullback square:
\begin{equation}\label{eq:log-condition} 
\xymatrix@-1pc{
\alpha^{-1}((A,\cdot)^{\times}) \ar[d] \ar[r] & (A,\cdot)^{\times} \ar[d] \\
M \ar[r]^{\alpha} & (A,\cdot)}
\end{equation}
\begin{definition}
A pre-log structure $(M,\alpha)$ on a simplicial commutative ring $A$ is a
\emph{log structure} if the top horizontal map in the square~\eqref{eq:log-condition} 
is a weak equivalence in the standard model structure on $s\mathcal M$. In this case,
$(A,M,\alpha)$ is called a \emph{log simplicial ring}. 
\end{definition}

\begin{corollary}
  If $(A,M) \to (B,N)$ is a weak equivalence of pre-log simplicial
  rings, then $(A,M)$ is a log simplicial ring if and only if $(B,N)$
  is.
\end{corollary}
\begin{proof}
This uses that the inclusion of path components is a fibration of simplicial sets.
\end{proof}
\begin{remark}
  While a pre-log simplicial ring is the same as simplicial object in
  the category of pre-log rings, it is not true that a log simplicial
  ring is a simplicial object in the category of log rings: Already in
  simplicial degree $0$, the monoid $((A,\cdot)^{\times})_0$ does not
  need to coincide with its submonoid $((A,\cdot)_0)^{\times}$. The
  homotopy invariance statement of the previous corollary would not
  hold if the log condition was defined using the degree-wise units.
\end{remark}
\begin{construction}\label{constr:logification}
  If $(A,M)$ is a pre-log simplicial ring, then we may factor the top
  horizontal map in the square~\eqref{eq:log-condition} as a
  cofibration $\alpha^{-1}((A,\cdot)^{\times}) \to G$ followed by an
  acyclic fibration $G \to (A,\cdot)^{\times}$ with respect to the
  standard model structure. The pushout $M^a = M
  \coprod_{\alpha^{-1}((A,\cdot)^{\times})}G$ of the resulting diagram
  in $s\mathcal M$ comes with a canonical map $\alpha^a\colon M^a \to
  (A,\cdot)$.

  The induced map $(\alpha^a)^{-1}((A,\cdot)^{\times}) \to (A,\cdot)^{\times}$ is
  isomorphic to $G \to (A,\cdot)^{\times}$. Hence $(M^a,
  \alpha^a)$ is a log structure on $A$.  We call it the
  \emph{associated log structure} of $(M,A)$ and refer to
  $(A,M^a,\alpha^a)$ as the \emph{logification} of $(A,M,\alpha)$.

  The logification comes with a natural map $(A,M,\alpha) \to
  (A,M^a,\alpha^a)$.  The use of the relative cofibrant replacement of
  $(A,\cdot)^{\times}$ and the left properness of $s\mathcal M$ ensures
  that logification preserves weak equivalences.
\end{construction}
\begin{lemma}
  If $(A, M,\alpha)$ is a log simplicial ring, then $(A,M,\alpha) \to
  (A,M^a,\alpha^a)$ is a weak equivalence.
\end{lemma}
\begin{proof}
  If $(M,\alpha)$ is a log structure, then
  $\alpha^{-1}((A,\cdot)^{\times}) \to G$ is a weak equivalence. This
  implies that $(A,M,\alpha) \to (A,M^a,\alpha^a)$ is a weak
  equivalence.
\end{proof}

\subsection{The log model structure}\label{subsec:log-model}
Our next aim is to express the log condition and the logification 
in terms of model structures. 

\begin{lemma}\label{lem:log-by-lifting-property}
Let $(A,M)$ be fibrant in the injective pre-log model structure.  Then  $(A,M)$ is a log ring if and only 
if for every cofibration $K \to L$ in $s\mathcal M$ with $L$ grouplike,  every commutative 
square 
\[\xymatrix@-1pc{ K \ar[r] \ar[d] & M \ar[d] \\ L \ar[r] & (A,\cdot)} \]
in $s\mathcal M$ admits a lift $L \to M$ making both triangles 
commutative.  \end{lemma}
\begin{proof}
Let $(A,M)$ be a pre-fibrant log simplicial ring. Then $L \to (A,\cdot)$ factors through
the inclusion $(A,\cdot)^{\times} \to (A,\cdot)$ because $L$ is grouplike, and there exists a lifting
in the resulting square 
\begin{equation}\label{eq:log-by-lifting-property}\xymatrix@-1pc{ K \ar[r] \ar[d] & \alpha^{-1}((A,\cdot)^{\times}) \ar[d] \\ L \ar[r] & (A,\cdot)^{\times}} \end{equation}
because $K \to L$ is cofibration and $\alpha^{-1}((A,\cdot)^{\times}) \to (A,\cdot)^{\times}$ is an acyclic
fibration. Composing with $ \alpha^{-1}((A,\cdot)^{\times}) \to M$ gives the desired lift. 

For the converse, it is enough to show that for every generating
cofibration $K \to L$ in the standard model structure on $s\mathcal M$
and every square of the form~\eqref{eq:log-by-lifting-property} there
exists a lift $L \to \alpha^{-1}((A,\cdot)^{\times})$. Since
$(A,\cdot)^{\times}$ is grouplike, the map $L\to (A,\cdot)^{\times}$
extends over the group completion $L \to L^{\gp}$. The
composed map $K \to L^{\gp}$ lifts against $M \to
(A,\cdot)$. This provides a map $L^{\gp} \to M$ whose
composite with $L\to L^{\gp}$, in combination with $L \to (A,\cdot)^{\times}$, induces the desired lifting
$L \to \alpha^{-1}((A,\cdot)^{\times})$.
\end{proof}

Every map $K \to L$ in $s\mathcal M$ gives rise to a pre-log simplicial ring $(\mathbb Z[L],K)$ 
and a canonical map $(\mathbb Z[L],K)\to(\mathbb Z[L],L)$ in $s\mathcal P$. An adjunction argument
shows
\begin{lemma}\label{lem:comparing-lifting-conditions}
Let $K \to L$ be a map in $s\mathcal M$, let $(A,M) \to (B,N)$ be a map in $s\mathcal P$, and
consider commutative squares
\[\xymatrix@-1pc{(\mathbb Z[L],K) \ar[r] \ar[d] &(A,M)\ar[d] \\(\mathbb Z[L],L)\ar[r] &(B,N)}
\qquad\text{and}\qquad\xymatrix@-1pc{ K \ar[r] \ar[d] & M \ar[d] \\ L \ar[r] & N\times_{(B,\cdot)}(A,\cdot)}
\]
in $s\mathcal P$ and  $s\mathcal M$. Then the universal property of $\mathbb Z[-]$ induces a  one-to-one correspondence between commutative squares of the first and second type, and the first square admits a lift 
if and only if the second does. \qed
\end{lemma}

Let $I$ be the set of generating cofibrations for the standard model structure on $s\mathcal M$,
and let  
\[ S = \{(\mathbb Z[L^{\gp}],K)\to(\mathbb Z[L^{\gp}],L^{\gp}) \;|\; (K \to L^{\gp}) = (K \xrightarrow{f} L \to L^{\gp}) \; \text{ where }\; f \in I\} \]
be set of maps in $s\mathcal P$ obtained by group-completing the codomains of the generating
cofibrations for $s\mathcal M$ and forming the associated maps of pre-log simplicial rings.  

We will say that a map of pre-log simplicial rings is a \emph{log
  equivalence} if it induces a weak equivalence after logification,
and a \emph{log cofibration} if it is a cofibration in the injective pre-log
model structure of Proposition~\ref{prop:inj-pre-log-model-str}. Moreover,
a pre-log simplicial ring is \emph{log fibrant} if it is a pre-fibrant log
simplicial ring.
\begin{theorem}\label{thm:log-model-str}
  The log equivalences and the log cofibrations are the weak
  equivalences and cofibrations of a left proper simplicial cellular
  \emph{log model structure} on the category of simplicial pre-log
  rings $s\mathcal P$. The log fibrant objects are the fibrant objects
  in this model structure.
\end{theorem}
We write $s\mathcal L$ for this model category. By slight abuse of
language, we refer to it as the model category of \emph{log simplicial
  rings}.
\begin{proof}
  The log model structure is defined to be the left Bousfield
  localization of the injective pre-log model structure with respect to $S$.
  Its existence and most of its properties are provided
  by~\cite[Theorem
  4.1.1]{Hirschhorn_model}. Lemma~\ref{lem:log-as-local-objects}
  provides the characterization of the fibrant objects, and
  Lemma~\ref{lem:logification-local-equivalence} and~\cite[Theorem
  3.2.18]{Hirschhorn_model} provide the characterization of the weak
  equivalences.
\end{proof}

\begin{lemma}\label{lem:log-as-local-objects}
A pre-log simplicial ring $(A,M)$ is $S$-local if and only if it is a pre-fibrant log simplicial ring.
\end{lemma}
\begin{proof}
  Let $(A,M)$ be a pre-fibrant log simplicial
  ring. By~\cite[Proposition 4.2.4]{Hirschhorn_model}, showing that it
  is $S$-local is equivalent to showing that $(A,M) \to *$ has the
  right lifting property with respect to the pushout product map
  \[ (\mathbb Z[L^{\gp}],L^{\gp})\otimes \partial \Delta^n
  \textstyle\coprod_{ (\mathbb Z[K],L^{\gp})\otimes \partial \Delta^n}
  (\mathbb Z[K],L^{\gp})\otimes \Delta^n \to (\mathbb
  Z[L^{\gp}],L^{\gp})\otimes \Delta^n.\] This map is isomorphic to the
  map $(\mathbb Z[L'],K')\to(\mathbb Z[L'],L')$ associated with
  \[ K' = L^{\gp}\otimes \partial \Delta^n \textstyle\coprod_{
    K\otimes \partial \Delta^n} L^{\gp}\otimes \Delta^n \to
  L^{\gp}\otimes \Delta^n = L'.\] Then $L'$ is grouplike because
  $L^{\gp}$ is grouplike and $\Delta^n$ is contractible. Combining
  Lemmas~\ref{lem:log-by-lifting-property} and
  \ref{lem:comparing-lifting-conditions} provides the desired lifting.

  Now assume that $(A,M)$ is $S$-local. Then $(A,M)$ is pre-fibrant by
  definition. Lemma~\ref{lem:comparing-lifting-conditions} and the
  argument given in the proof of
  Lemma~\ref{lem:log-by-lifting-property} show that $(A,M)$ is a log
  simplicial ring.
\end{proof}
The next lemma exhibits the logification of Construction~\ref{constr:logification} as an explicit
fibrant replacement for the log model structure.  
\begin{lemma}\label{lem:logification-local-equivalence}
Let $(A,M)$ be a pre-fibrant pre-log simplicial ring. Then the logification map $(A,M) \to (A,M^a)$ is 
an $S$-local equivalence of pre-log simplicial rings. 
\end{lemma}
\begin{proof}
  Let $\alpha^{-1}((A,\cdot)^{\times}) \to G$ be the cofibration used
  in Construction~\ref{constr:logification}.  Then we can form the
  associated map of pre-log simplicial rings and observe that the
  logification may be obtained as the right vertical map in the
  pushout square
\[\xymatrix@-1pc{
(\mathbb Z[G],  \alpha^{-1}((A,\cdot)^{\times})) \ar[r] \ar[d] & (A,M) \ar[d] \\
(\mathbb Z[G],  G) \ar[r] & (A,M^a). 
}
\] 
It is enough to show that the left hand vertical map is an $S$-local
equivalence. For this we have to verify that it induces a weak
equivalence of simplicial sets when applying the functor
$\bfHom(-,(B,N))$ where $(B,N)$ is a fibrant object in the log
model structure. By adjunction and
Lemma~\ref{lem:comparing-lifting-conditions}, this is equivalent to
showing that
\[ \alpha^{-1}((A,\cdot)^{\times}) \otimes \Delta^n
\textstyle\coprod_{\alpha^{-1}((A,\cdot)^{\times})
  \otimes \partial\Delta^n} G \otimes \partial \Delta^n \to G \otimes
\Delta^n \] has the lifting property against $N \to (B,\cdot)$. Since
$(B,N)$ is log and $G$ is grouplike, this follows from
Lemma~\ref{lem:log-by-lifting-property}.
\end{proof}
The last lemma and the formal properties of a left Bousfield localization easily imply
the following statement. 
\begin{corollary} \label{bousfieldexplained} The homotopy category
  $\mathrm{Ho}(s\mathcal{L})$ is equivalent to the full subcategory of
  $\mathrm{Ho}(s\mathcal{P})$ consisting of log simplicial rings, and
  the logification induces an adjoint pair $(-)^{a}
  :\mathrm{Ho}(s\mathcal{P})\rightleftarrows \mathrm{Ho}(s\mathcal{L})
  : i$ where $i$ is the canonical inclusion functor.
\end{corollary}

\subsection{The replete model structures}
Rognes' notion of \emph{repletion} discussed in Section~\ref{subsec:repletion} can also
be described in terms of appropriate model structures on $s\mathcal P$. 
\begin{proposition}\label{prop:replete-inj-model}
  The category of simplicial pre-log rings $s\mathcal P$ admits a left proper
  simplicial \emph{replete pre-log} model structure where
\begin{itemize}
\item a map $(f,f^{\flat}) \colon (A,M) \to (B,N)$ is cofibration if
  $f$ is a cofibration in $s\mathcal R$ and $f^{\flat}$ is a
  cofibration in $s\mathcal M$ and 
\item an object $(A,M)$ is fibrant if $A$ is fibrant in $s\mathcal R$,
  the structure map $M \to (A,\cdot)$ is a fibration in the standard
  model structure on $s\mathcal M$, and $M$ is grouplike.
\end{itemize}
The forgetful functor $s\mathcal P \to s\mathcal M$ sending $(A,M)$ to
$M$ is a right Quillen functor with respect to the replete pre-log
model structure and the group completion model structure on $s\mathcal
M$.
\end{proposition}
\begin{proof}
  We let $\xi \colon C \to C'$ be the map in $s\mathcal M$ introduced
  in \eqref{eq:definition-xi} and form the left Bousfield localization
  of the injective pre-log model structure with respect to the
  associated map $(\mathbb Z[\xi],\xi)$ in $s\mathcal P$. This model
  structure has the same cofibrations as the injective pre-log model
  structure, and the isomorphism $\bfHom_{s\mathcal P}((\mathbb
  Z[\xi],\xi), (A,M)) \cong \bfHom_{s\mathcal M}(\xi,M)$ shows that the
  fibrant objects are the pre-fibrant objects $(A,M)$ with $\xi$-local
  $M$. Hence Lemma~\ref{lem:char-local} provides the characterization
  of the fibrant objects.  Since $M \mapsto (\mathbb Z [M],M)$ is left
  adjoint to the forgetful functor $s\mathcal P \to s\mathcal M$, the
  last statement is a formal consequence of~\cite[Theorem
  3.3.20]{Hirschhorn_model}.
\end{proof}
If $(A,M)$ is a pre-log simplicial ring, then the group completion of
$M$ enables us to form the \emph{trivial locus} $(A[M^{-1}],M^{\gp})=(\mathbb
Z[M^{\gp}]\otimes_{\mathbb Z[M]} A, M^{\gp})$. Up to a pre-fibrant
replacement, this construction can be viewed as a fibrant replacement
in the replete pre-log model structure:

\begin{lemma}\label{lem:trivial-locus-fib-repl} The composite $(A,M) \to (A[M^{-1}], M^{\gp})^{\mathrm{pre-fib}}$ of
  the canonical map $(A,M) \to (A[M^{-1}], M^{\gp})$ with a fibrant
  replacement functor for the injective pre-log model structure
  provides a fibrant replacement functor for the replete pre-log model
  structure.
\end{lemma}
\begin{proof}
  Since $i\colon M \to M^{\gp}$ is an acyclic cofibration in the group
  completion model structure, the associated map $(\mathbb Z[i],i)$ is
  an acyclic cofibration in the replete pre-log model structure.  The map
  $(A,M) \to (A[M^{-1}], M^{\gp})$ is
  a cobase change of this map and hence also an acyclic cofibration in
  the replete pre-log model structure. This implies that the map in question
  is an acyclic cofibration whose codomain is fibrant in the replete pre-log
  model structure.
\end{proof}

As it is often the case with left Bousfield localizations, we don't
have an explicit characterization of general fibrations in the
replete pre-log model structure. However, the replete pre-log model
structure can be used to guarantee exactness on the underlying monoid
map of a fibrant augmented object:
\begin{corollary}
  Let $(A,M)$ be a pre-log simplicial ring and let $s\mathcal
  P_{(A,M)}/{(A,M)}$ be the category of augmented $(A,M)$-algebras
  with the model structure induced by the replete pre-log model
  structure on $s\mathcal P$. If $(A,M) \to (B,N)\to (A,M)$ is
  fibrant in this model category, then the underlying map $N
  \to M$ is exact in the sense of
  Definition~\ref{def:virt-surj-and-exact}.
\end{corollary}
\begin{proof}
  By Proposition~\ref{prop:replete-inj-model}, the map $N \to M$ is a
  fibration in the group completion model structure. Hence $N \to
  N^{\mathrm{rep}}$ is a weak equivalence in the standard model
  structure, and since $N \to M$ is virtually surjective,
  Proposition~\ref{prop:exact-replete}(iii) shows that $N \to M$ is
  exact.
\end{proof}
\begin{remark}By analogy with~\cite[Definition 3.12]{rog-log}, one can define the
repletion of a map $(B,N) \to (A,M)$ of pre-log simplicial rings with
virtually surjective $N \to M$ as the first map in the factorization
\[ (B,N) \to (\mathbb Z[N^{\mathrm{rep}}]\otimes_{\mathbb Z[N]}B,
N^{\mathrm{rep}}) \to (A,M). \] A similar argument as in
Lemma~\ref{lem:trivial-locus-fib-repl} shows that the repletion map is
an acyclic cofibration in the replete pre-log model structure.  However, we do
not know if the fact that $N^{\mathrm{rep}} \to M$ is a fibration in the group
completion model structure is sufficient to
conclude that $(\mathbb Z[N^{\mathrm{rep}}]\otimes_{\mathbb Z[N]}B,
N^{\mathrm{rep}}) \to (A,M)$ gives rise to a fibration in the replete
pre-log model structure after replacing it by a fibration of pre-log
simplicial rings.
\end{remark}
\begin{remark}
The projective pre-log model structure gives rise to a projective version of the replete pre-log model structure with similar properties.
\end{remark}
\begin{remark}
  Combining the arguments of Proposition~\ref{prop:replete-inj-model}
  with the log model structure of Theorem~\ref{thm:log-model-str}, we
  obtain a left proper simplicial \emph{replete} log model structure
  on $s\mathcal P$. Here an object is fibrant if and only if it is
  injectively fibrant as a pre-log simplicial ring, $M \to (A,\cdot)$
  is a log structure, and $M$ is grouplike. 

  It follows that the fibrant objects in this model structure always
  carry the trivial log structure. Up to pre-log fibrant replacement,
  the fibrant replacement of $(A,M)$ in the replete log model
  structure is given by $(A,M)\to(A[M^{-1}], A[M^{-1}]^{\times})$.
\end{remark}
\subsection{Functorialities}

\begin{definition}
  \label{defn:inverse_image_log_str}
  Let $(A,M)$ be a simplicial pre-log ring, and let $f \colon A \to 
  B$ be a morphism of simplicial rings. Then the \emph{inverse image 
    pre-log structure} on $B$ is given by $M \to (A,\cdot) \to 
  (B,\cdot)$ and is denoted by $f_* M$. The \emph{inverse image 
    log structure} is defined to be the associated log structure. We 
  will denote it by $(f_* M)^a \to (B,\cdot)$.
\end{definition}

\begin{definition}
  \label{defn:direct_image_log_str}
  Let $(B,N)$ be a simplicial pre-log ring, and let $f \colon A \to 
  B$ be a morphism of simplicial rings. Then the \emph{direct image 
    pre-log structure} on $A$ is given by the fiber product of 
  simplicial monoids
  \[
    \xymatrix@-1pc{
      f^*N \ar[r] \ar[d] & (A,\cdot) \ar[d] \\
      N \ar[r] & (B, \cdot).
    }
  \]
  The associated log structure is denoted by $(f^*N)^a$.
\end{definition}
It is straightforward to check that if $(B,N)$ is a log simplicial
ring, then $f^*N$ is again a log structure on $(A,M)$ if $ N \to
(B,\dot)$ or $A^{\times} \to B^{\times}$ is a fibration. On the
contrary, the inverse image of a log structure will in general not
again be a log structure.

\begin{definition}
  Let $(f,f^{\flat}) \colon (A,M) \to (B,N)$ be a morphism of log simplicial 
  rings. Then $(f,f^{\flat})$ is \emph{strict} if $(f_*M)^a \to 
  N$ is an equivalence of simplicial monoids.
\end{definition}

If $A$ is a simplicial commutative ring, we will denote by
$s\mathcal{P}_{A}$ the category of pre-log structures on $A$, i.e., the
over-category $s\mathcal{M}/(A,\cdot)$, with its canonical induced
injective model structure. Likewise, we denote by $\mathbf{sLog}_{A}$
the full subcategory of the homotopy category of $s\mathcal{P}_{A}$ on
the objects $M \to (A,\cdot)$ that are log structures.

\begin{proposition}
A morphism of simplicial commutative rings $f \colon A\to B$ induces a Quillen adjunction \[f_{*}\colon s\mathcal{P}_{A} \rightleftarrows s\mathcal{P}_{B} : f^*.\]
On the level of homotopy categories, this adjunction and the logification 
induce an adjunction  
\[f_{*a}\colon \mathbf{sLog}_{A} \rightleftarrows \mathbf{sLog}_{B} : f^{*}=f^{*a}\] 
whose left adjoint $f_{*a}$ sends $M \to (A,\cdot)$ to $(f_*M)^a \to B$. 
\end{proposition}

\begin{proof} The first adjunction is immediate, and it is easy to verify that $f_*$ preserves cofibrations and trivial cofibrations. The second adjunction follows using Corollary~\ref{bousfieldexplained}.
\end{proof}

\begin{remark} The adjunction $(f_*,f^*)$ actually induces the structure of a \emph{left Quillen presheaf} over $s\mathcal R$ on $s\mathcal{P}$, endowed with the injective pre-log model structure (see for example~\cite[p. 127]{Simpson_ht-hc} for the notion of left Quillen presheaf).
\end{remark}

\section{Log-derivations and the log cotangent complex}\label{sec:log-der}
\subsection{Log-derivations}

We begin by defining derivations in the pre-log context. For this we
use that the simplicial model structures discussed in the previous
section provide simplicial mapping spaces for the respective
categories, and we will write $\Map_{\mathcal C}(-,-)$ for the derived
mapping spaces in a simplicial model category $\mathcal C$.

Let $s \sP _{(A,M)//(B,N)}$ denote the category of simplicial pre-log 
$(A,M)$-algebras over $(B,N)$.

\begin{definition}
  Let $(A,M) \rightarrow (B,N)$ be a morphism of simplicial pre-log 
  rings, and let $J$ be a simplicial $B$-module. Denote by $B \oplus J$ 
  the trivial square zero extension of $B$ by $J$, and let $N\oplus J$ 
  be the simplicial monoid $N \times J^{\textrm{add}}$. Define a 
  morphism $N\oplus J\rightarrow (B\oplus J, \cdot)$ as the product of 
  the two canonical maps $$N\longrightarrow (B,\cdot)\longrightarrow 
  (B\oplus J, \cdot),$$ $$(J,+)\hookrightarrow (B\oplus J, \cdot),\quad 
  x\mapsto (1,x).$$ Then  $(B\oplus J, N\oplus J)$ is canonically an 
  object in $s\mathcal{P}_{(A,M)//(B,N)}$, and we will call it the 
  \emph{trivial square zero extension of $(B,N)$ by $J$}.
\end{definition}

\begin{remark}
  \label{rem:EquivDefDer}
  In case $(f,f^{\flat}) \colon (A,M) \to (B,N)$ is a morphism of log 
  simplicial rings, an equivalent definition of the trivial square zero 
  extension is $(B \oplus J, ((s_0)_*N)^a)$, where $s_0 \colon B \to B 
  \oplus J$ is the canonical section of the projection $B \oplus J \to 
  B$. See \cite[Lemma 11.5]{rog-log}.
\end{remark}
\begin{definition}
  \label{def:derivations}
  Let $(f,f^{\flat})\colon (A,M) \to (B,N)$ be a morphism of simplicial 
  pre-log rings, and let $J$ be a simplicial $B$-module. The simplicial 
  set of $f$-linear derivations of $(B,N)$ with values in $J$ is defined 
  as
  \[
    \Der_{(A,M)}((B,N),J) = \Map_{s\mathcal{P}_{(A,M)//(B,N)}} ((B,N),(B 
    \oplus J, N \oplus J)).
  \]
\end{definition}

For a morphism of log simplicial rings, it does not make a difference if 
we compute derivations in the category of log simplicial rings or in the 
category of simplicial pre-log rings:

\begin{lemma}
  \label{lem:PLDerEqLDer}
  Let $(f,f^{\flat}) \colon (A,M) \to (B,N)$ be a morphism of log 
  simplicial rings. Then
  \[
    \Der_{(A,M)}((B,N),J) \simeq \Map_{s\mathcal{L}_{(A,M)//(B,N)}} 
    ((B,N),(B \oplus J, N \oplus J))
  \]
\end{lemma}
\begin{proof}
  The map $ (B \oplus J, N \oplus J) \to (B,N)$ is a fibration in the
  projective pre-log model structure since $J$ is fibrant as a
  simplicial set. So we can model the derived mapping space
  $\Map_{s\mathcal{P}_{(A,M)//(B,N)}} ((B,N),(B \oplus J, N \oplus
  J))$ by the simplicial mapping space in $s \sP _{(A,M)//(B,N)}$
  where we use $(B \oplus J, N \oplus J)$ as the target and a
  cofibrant replacement of $(A,M) \to (B,N) \xrightarrow{=} (B,N)$ in
  the projective pre-log model structure as the source. Since
  cofibrations in the projective pre-log model structure are also
  cofibrations in the injective pre-log model structure, it remains to
  show that with respect to the injective pre-log model structure, the
  target is weakly equivalent to a fibration in the log model
  structure.  By Remark \ref{rem:EquivDefDer}, $(B \oplus J, N \oplus
  J) \simeq (B \oplus J, ((s_0)_* N)^a)$, showing that $(B \oplus J, N
  \oplus J)$ is a log simplicial ring. Using~\cite[Proposition
  3.3.16]{Hirschhorn_model}, it follows that the fibrant replacement
  of $ (B \oplus J, N \oplus J)$ in the injective pre-log model
  structure on $s \sP _{(A,M)//(B,N)}$ also provides a fibrant
  replacement in the log model structure. 
\end{proof}

\begin{remark}
  \label{remark:strict}
  In case $(f,f^{\flat}) \colon (A,M) \to (B,N)$ is a morphism of log 
  simplicial rings, every log-derivation is a strict morphism. One can 
  show that for a morphism of discrete log rings, the functor from 
  $B$-modules to trivial square zero extensions of $(B,N)$ gives an 
  equivalence of categories between abelian objects in the category of 
  log $(A,M)$-algebras that are strict over $(B,N)$ and the category of 
  $B$-modules \cite[Lemma 4.13]{rog-log}. The same should also hold for 
  a morphism of log simplicial rings, giving a Quillen equivalence 
  between the categories of simplicial $B$-modules and the category $(s 
  \sP ^{\str}_{(A,M)//(B,N)})_{\ab}$ of abelian objects in the category 
  of log $(A,M)$-algebras that are strict over $(B,N)$.
\end{remark}

\subsection{The log cotangent complex}
We have a functor
\[
\Omega \colon s \sP _{(A,M)//(B,N)} \longrightarrow \Mod _B, \qquad
(C,O) \longmapsto \Omega_{(C,O)/(A,M)} \otimes _C B
\]
where $\Omega_{(C,O)/(A,M)}$ is defined by level-wise application of
the functor of log K\"ahler differentials (see \cite[Section 1.7]{ka} 
for the definition) for discrete pre-log rings.
On the other hand, we have the functor of the previous section \[ K
\colon \Mod_B \longrightarrow s \sP _{(A,M)//(B,N)}, \qquad 
J \longmapsto (B \oplus J, N \oplus J).
\]
Note that $K$ is given by applying the trivial square zero
extension functor for discrete pre-log rings levelwise.  We then have the
following result:
\begin{lemma}
  The pair $\Omega \colon s \sP _{(A,M)//(B,N)} \rightleftarrows
  \Mod_B : K$ is a Quillen adjunction with respect to the projective 
  pre-log model structure on $\sP _{(A,M)//(B,N)}$ and the standard model
  structure on $\Mod_B$.
\end{lemma}
\begin{proof}
  Adjointness follows from the corresponding statement for discrete
  log rings, since $\Omega$ and $K$ are both applied level-wise.
  Since $K$ clearly preserves fibrations and trivial fibrations, the
  adjunction is in fact a Quillen adjunction.
\end{proof}
Since $\Omega$ is part of a Quillen adjunction we obtain a left
derived functor
\[
L \Omega \colon \Ho (s \sP _{(A,M)//(B,N)}) \longrightarrow \Ho
(\Mod_B).
\]
\begin{definition}
  We define the log cotangent complex $\IL_{(B,N)/(A,M)}$ of a
  morphism of simplicial pre-log rings $(A,M) \to (B,N)$ to be $L \Omega
  (B,N)$.  Here $(B,N)$ is regarded as an object of $s \sP
  _{(A,M)//(B,N)}$.
\end{definition}
Thus by definition, the log cotangent complex represents the
derivations, since by adjunction we have
\begin{align*}
  \Der_{(A,M)}((B,N),J) &= \Map_{s\mathcal{P}_{(A,M)//(B,N)}}
  ((B,N),(B \oplus J, N \oplus J))\\
  &\simeq \Map_{\Mod_B}(\IL_{(B,N)/(A,M)},J)
\end{align*}
\begin{remark}
  In case $(A,M) \to (B,N)$ is a morphism of discrete log rings, the
  above definition recovers Gabber's definition~\cite[\S 8]{olsson} of the log cotangent
  complex.
\end{remark}
\begin{remark}
  \label{remark:LeftDerivedAb}
  Note that by Remark \ref{remark:strict} the log K\"ahler differentials 
  of a morphism of discrete log rings only depend on the abelian objects 
  in the category of log rings that are strict over $(B,N)$, since the 
  log K\"ahler differentials explicitly compute the abelianization 
  functor in this category. If we assume that we have the Quillen 
  equivalence between $\sMod_B$ and $(s \sP 
  ^{\str}_{(A,M)//(B,N)})_{\ab}$ mentioned in Remark 
  \ref{remark:strict}, then by Lemma \ref{lem:PLDerEqLDer} the log 
  cotangent complex explicitly computes the left derived of 
  abelianization in the category $s \sP ^{\str}_{(A,M)//(B,N)}$. As 
  Rognes points out in \cite{rog-log}, it should be interesting to 
  investigate the abelianization functor in other categories than log 
  rings that are strict over $(B,N)$. For instance one could replace 
  strictness by the weaker notion of repleteness, or by no condition at 
  all.
\end{remark}

The homotopical version of the log cotangent complex used in~\cite[\S 11]{rog-log} immediately
translates to give the following version of the log cotangent complex of a morphism
of log simplicial rings $(f, f^{\flat}) \colon (A,M) \to (B,N)$, which 
is defined as the following homotopy pushout:
\[
\xymatrix@-1pc{ B \otimes _{\IZ[N]} \IL_{\IZ[N]/\IZ[M]} \ar[r]^-{\psi}
  \ar[d]_{\phi} & B \otimes N^{\gp}/M^{\gp} \ar[d]^-{\bar{\phi}}\\
  \IL_{B/A} \ar[r]_{\bar{\psi}} & \IL^{\Rog}_{(B,N)/(A,M)} }
\]
Here $\psi$ is in simplicial degree $s$ defined as the application of the morphism
\begin{align*}
  B_s \otimes _{\IZ[N_s]} \Omega^1_{\IZ[N_s]/\IZ[M_s]} &\to B_s \otimes 
  N_s^{\gp}/M_s^{\gp}\\
  b \otimes dn &\longmapsto b \cdot \beta(n) \otimes \gamma (n)
\end{align*}
where $\gamma$ denotes the canonical morphism to group completion.

Rognes' verification that this complex represents the derived functor of
derivations also carries over to the present context:
\begin{proposition}
  \cite[Proposition 11.21]{rog-log} There is a natural weak
  equivalence of mapping spaces
  \[
  \Map_{\Mod_B}(\IL^{\Rog}_{(B,N)/(A,M)},J) \simeq
  \Der_{(A,M)}((B,N),J)
  \]
\end{proposition}

This allows us to compare the two definitions.
\begin{theorem}\label{thm:comparison-rognes-cot-cx}
  Let $(f, f^{\flat}) \colon (A,M) \to (B,N)$ be a morphism of
  simplicial pre-log rings. Then
  \[
  \IL_{(B,N)/(A,M)} \cong \IL^{\Rog}_{(B,N)/(A,M)}
  \]
  in $\mathrm{D}(B)$.
\end{theorem}

\begin{proof}
  Let $J$ be a simplicial module. Then the functor mapping $J$ to the 
  logarithmic derivations $\Der_{(A,M)}((B,N),J)$ is representable both 
  by $\IL_{(B,N)/(A,M)}$ and $\IL^{\Rog}_{(B,N)/(A,M)}$. Using the 
  Yoneda lemma, we thus conclude that $  \IL_{(B,N)/(A,M)} \cong 
  \IL^{\Rog}_{(B,N)/(A,M)} $ in the derived category $D(B)$ of 
  simplicial $B$-modules.
\end{proof}

Rognes definition leads to simple proofs of the expected properties of 
the log cotangent complex.

\begin{proposition}\label{basictriangles}
  \begin{enumerate}[(i)]
  \item Let $(A,M) \rightarrow (B,N) \rightarrow (C,O)$ be maps of 
    simplicial pre-log rings. Then there is a transitivity homotopy
  cofiber sequence in the homotopy category of simplicial $C$-modules
  \[C\otimes_{B}^{\mathbb{L}}\mathbb{L}_{(B,N)/(A,M)} \longrightarrow 
  \mathbb{L}_{(C,O)/(A,M)} \longrightarrow \mathbb{L}_{(C,O)/(B,N)}.\]
  \item Let \[\xymatrix@-1pc{ (A,M) \ar[d] \ar[r] & (B,N) \ar[d] \\ (R,P) \ar[r]
    & (S,Q)}\] be a homotopy pushout square in $s\mathcal{L}$, then
  there is an isomorphism in the homotopy category of simplicial
  $S$-modules \[S\otimes_{B}^{\mathbb{L}}\mathbb{L}_{(B,N)/(A,M)}
  \simeq \mathbb{L}_{(S,Q)/(R,P)}.\]
\end{enumerate}
\end{proposition}

\begin{proof} These follow immediately from \cite[Propositions 11.28 and 11.29]{rog-log}.
\end{proof}

\subsection{Square zero extensions}

\begin{definition}\label{def:square}
Let $(f,f^\flat)\colon (A,M) \to (B,N)$ be a morphism of log simplicial rings, $J$ 
a simplicial $B$-module, and $\eta \colon \Omega_{(B,N)/(A,M)} \to J$ a 
derivation. We define $(B^{\eta}, N^{\eta})$ via the pullback diagram 
in $s\mathcal{L}_{(A,M)}$
\[ \xymatrix@-1pc{(B^{\eta}, N^{\eta}) \ar[r] \ar[d]_-{p_{\eta}} & 
  (B,N)\ar[d]^-{s_0} \\ (B,N) \ar[r]^-{\eta} & (B \oplus J, N \oplus 
  J)}\] and call the map $p_{\eta} \colon (B^{\eta}, N^{\eta}) 
\rightarrow (B,N)$ the \emph{natural projection}.
\end{definition}

This defines a functor
\[
  \Phi \colon (\Mod_B)_{\Omega_{(B,N)/(A,M)}} \to s \sL_{(A,M)//(B,N)}
\]
from the category of simplicial modules under the log K\"ahler 
differentials to the category of log simplicial $(A,M)$-algebras 
augmented over $(B,N)$. This functor has a left adjoint $\Psi$ given by 
mapping an object $(A,M) \to (C,O) \to (B,N)$ to the sequence of 
differentials $\Omega_{(B,N)/(A,M)} \to \Omega_{(B,N)/(C,O)}$. It is 
straightforward to verify that this defines a Quillen adjunction, so 
that we obtain derived functors $L\Psi$ and $R\Phi$.

The following statement is proved in Appendix~\ref{sec:appendixA}:
\begin{proposition}\label{prop:cruciallemma} Let $\pi\colon R \rightarrow 
  S$ be a square zero extension of discrete commutative rings, and let 
  $J= \ker \pi$ be the corresponding square zero 
  ideal. Then there exists a derivation $d \in 
  \pi_{0}\Map_{R/s\mathcal{A}/S}(S,S \oplus J[1])$ such that 
  there exists an isomorphism in $\mathrm{Ho}(s\mathcal{R}/S)$, 
  between $\pi\colon R \rightarrow  S$ and the canonical projection 
  $p_d\colon S \oplus_d J \rightarrow S,$ where 
  $p_d$ is defined by the homotopy pullback diagram\[
  \xymatrix@-1pc{S \oplus_{d} J \ar[r] 
    \ar[d]_-{p_d} & S \ar[d]^-{0} \\ S 
    \ar[r]^-{d} & S \oplus J[1].}\] 
\end{proposition}
Theorem~\ref{thm:strict-sq-zero-log-derivation} below provides the
analog of Proposition~\ref{prop:cruciallemma} for strict square zero
log extensions of discrete log rings. The proof of that theorem will
be based on the next two lemmas. For this recall that a monoid $P$ is
\emph{integral} if the group completion map $P \to P^\gp$ is
injective.

The following elementary result is well known (see for
example the remark after~\cite[Definition (4.6)]{ka}).  We include a
proof since we were unable to locate one in the literature.
\begin{lemma}\label{lemma:strict-exact}   A strict square zero extension of discrete and integral log rings $(\pi, \pi^{\flat})\colon (R,P,\alpha) \rightarrow (S, Q,\beta)$ is exact, i.e., the diagram \[\xymatrix@-1pc{P\ar[d]_{\pi^{\flat}} \ar[rr] && P^{\gp} \ar[d]^-{(\pi^{\flat})^{\gp}} \\ Q \ar[rr] && Q^{\gp} }\] is cartesian in the category of commutative monoids.
\end{lemma}
\noindent If $(R,P,\alpha)$ is a discrete log ring, then we
write $c_R\colon R^\times \to P$ for the composite of the inverse of
the isomorphism $\alpha^{-1}(R^\times) \to R^\times$ and the canonical
map $\alpha^{-1}(R^\times) \to P$.
\begin{proof} Let $J = \ker \pi$.  Since $\pi$ is surjective and $J^2=0$, we have $\pi^{-1}(S^{\times}) \cong R^{\times}$. Together with the log condition on $(R,P,\alpha)$, this provides isomorphisms $ R^{\times} \cong \alpha^{-1}(R^{\times}) \cong \alpha^{-1} \pi^{-1} (S^{\times})$, and therefore strictness implies that we have a pushout of commutative monoids \[\xymatrix@-1pc{ R^{\times} \ar[rr]^-{c_R} \ar[d]_-{\pi^{\times}} && P \ar[d]^-{\pi^{\flat}} \\ S^{\times} \ar[rr]^-{c_{S}} && Q.}\]
Since both $R^{\times}$ and $S^{\times}$ are abelian groups, we have $Q \cong P \oplus S^{\times} /\!\sim$ where \[(x, u) \sim (x', u') \Leftrightarrow \exists \,v\in R^{\times} \text{ such that} \, c_R(v)x= x' \text{ and } \pi(v^{-1}) u = u'.\] Given $[x,u] \in P \oplus S^{\times} /\!\sim$, the square zero condition implies that there exists $v \in R^{\times}$ such that $\pi(v)=u$. Since $(x,u) \sim (c_R (v)x, 1)$, the morphism $\pi^{\flat}$ maps the element $c_R (v) x \in P$  to  $[x,u] \in Q \cong P \oplus S^{\times} /\!\sim$. Hence the morphism $\pi^{\flat}$ is surjective.

Consider the morphism of monoids \[\varphi\colon P \longrightarrow Q \times _{Q^{\gp}} P^{\gp}\qquad x \longmapsto (\pi^{\flat}(x), [x,1]_{\gp})\] where we have denoted by $x \mapsto [x, 1]_{\gp}$ the canonical map $P \to P^{\gp}$, with $P^{\gp}$ implicitly identified, as usual, with a quotient of $P \times P$.

  We will prove that $\varphi$ is bijective, hence a monoid isomorphism. Since $P$ is integral, the map $P \to P^{\gp}$ is injective, and thus so is $\varphi$. In order to prove surjectivity of $\varphi$, we first observe that if $(y, [x_1, x_2]_{\gp})$ is an arbitrary element in $Q \times_{Q^{\gp}} P^{\gp}$, then there exists $x \in P$ such that $\pi^{\flat}(x)= y$. It follows that $\pi^{\flat}(xx_2)= \pi^{\flat}(x_1)$ since $y$ and $[x_1, x_2]_{\gp}$ have the same image in $Q^{\gp}$, and $Q$ is integral. But under the identification $Q \cong P \oplus S^{\times} /\!\sim $, we have $\pi^{\flat}(x')= [x', 1]$ for any $x' \in P$, and hence $ [xx_2, 1]= [x_1, 1]$ in $P \oplus S^{\times}/\!\sim$. By definition of $\sim$, there exists $v \in R^{\times}$ such that \[c_R(v)xx_2= x_1 \quad\text{ and }\quad \pi(v^{-1}) 1= 1.\] In particular, we have $\pi(v)= 1$. Now, let us prove that $\varphi((c_R (v)x))= (y, [x_1, x_2])$. By definition, we have \[\varphi (c_R (v)x)= (\pi^{\flat}c_R (v) \pi^{\flat}(x), [c_R(v)x, 1]_{\gp}).\] Since $c_R(v)xx_2= x_1$, we have $[c_R(v)x, 1]_{\gp} = [x_1,x_2]_{\gp}$, and since $\pi^{\flat}c_R (v)= c_{S} \pi (v)= 1 \in Q$, we also have $\pi^{\flat}c_R (v) \pi^{\flat}(x) = \pi^{\flat}(x)=y$. This completes the proof of the surjectivity of $\varphi$. 
\end{proof}
If $(R,P)$ is a discrete log ring, then $c_R^{\gp}\colon R^\times \to P^\gp$ denotes the composite of the map $c_R\colon R^{\times} \to P$ from the log condition and the group completion map $P \to P^\gp$. 
\begin{lemma}\label{lemma:diginto} Let $(\pi,\pi^{\flat})\colon (R,P,\alpha)
  \rightarrow (S, Q,\beta)$ be a strict square zero extension of discrete integral
  log rings, and let $J = \ker\pi$ be the corresponding square zero
  ideal. Then \[\exp \colon J \longrightarrow R^{\times},\qquad \xi
  \longmapsto (1+\xi)\] and the maps $c_R^{\gp}$ and $c_{S}^{\gp}$
  induce a commutative diagram 
  \begin{equation}\label{eq:diginto}\xymatrix@-1pc{ J \ar[rr]^-{\exp} \ar[d] &&R^{\times} \ar[d]^-{\pi^{\times}} \ar[rr]^-{c_{R}^{\gp}} && P^{\gp} \ar[d] ^-{(\pi^{\flat})^\gp} \\
\{1\} \ar[rr] && S^{\times} \ar[rr]^-{c_{S}^{\gp}} && Q^{\gp}.   }
\end{equation}
of constant simplicial commutative monoids where both squares are
homotopy cartesian and homotopy cocartesian.
\end{lemma}

\begin{proof}
  It is immediate from the square zero condition that the left hand
  square is cartesian.  Since it consists of constant simplicial
  monoids, it is also homotopy cartesian. Hence $J \to R^{\times} \to
  P^{\gp}$ is a homotopy fiber sequence of simplicial abelian groups,
  and it follows that the left hand square is also homotopy
  cocartesian as a square of simplicial abelian groups. By
  Corollary~\ref{cor:hocart-sAb-sM}, it is therefore a homotopy cocartesian
  square of simplicial commutative monoids.

  For the right hand square, we consider the commutative diagram 
\begin{equation}\label{eq:exact-square-zero}\xymatrix@-1pc{ J \ar[rr]^-{\exp} \ar[d] &&R^{\times} \ar[d]^-{\pi^{\times}} \ar[rr]^-{c_{R}}&& P \ar[d]^-{\pi^{\flat}}\ar[rr] && P^{\gp} \ar[d] ^-{(\pi^{\flat})^\gp} \\
\{1\} \ar[rr] && S^{\times} \ar[rr]^-{c_{S}} && Q \ar[rr] &&Q^{\gp}.   }
\end{equation}
Since $(\pi,\pi^\flat)$ is a strict square zero extension between discrete integral log rings, Lemma~\ref{lemma:strict-exact} implies that the right hand square
in~\eqref{eq:exact-square-zero} is cartesian. The middle square in~\eqref{eq:exact-square-zero} is cartesian since the log
conditions on $(R,P,\alpha)$ and $(S,Q,\beta)$ and the square zero condition on
$\pi$ provide isomorphisms \[(\pi^\flat)^{-1}(S^\times) \cong
(\pi^\flat)^{-1}(\beta^{-1}(S^\times)) \cong
\alpha^{-1}\pi^{-1}(S^{\times}) \cong \alpha^{-1}(R^\times) \cong
R^{\times}.\] Hence the right hand square in~\eqref{eq:diginto} is
homotopy cartesian. Moreover, since we already observed that the left
hand square in~\eqref{eq:exact-square-zero} is cartesian, it follows
that the outer square is cartesian. Arguing as above, we deduce that
the outer square is homotopy cocartesian as a square of constant
simplicial commutative monoids. Since we already know that the left
hand square in~\eqref{eq:diginto} is homotopy cocartesian, it follows
that the right hand square in~\eqref{eq:diginto} is homotopy
cocartesian.
\end{proof}

We are now able to prove the promised log-analog of Proposition~\ref{prop:cruciallemma}.

\begin{theorem}\label{thm:strict-sq-zero-log-derivation}
Let $(\pi,\pi^\flat)\colon (R,P) \rightarrow (S, Q)$ be a strict square zero extension of discrete integral log rings, and let $J = \ker \pi$ be the corresponding square zero ideal. Then there exists a derivation \[(d,d^\flat) \in \pi_{0}\Map_{(R,P)/s\mathcal{P}/(S,Q)}((S,Q),(S \oplus J[1], Q \oplus J[1]))\] such that there exists an isomorphism in $\mathrm{Ho}(s\mathcal{P}/(S,Q))$ between $(\pi,\pi^\flat)$ and the canonical projection \[(p^{\phantom{\flat}}_d, p^\flat_d)\colon (S \oplus_{d} J, Q\oplus_{d^\flat} J)\rightarrow (S, Q),\] where $(p^{\phantom{\flat}}_d, p^\flat_d)$ is defined by the homotopy pullback diagram 
\begin{equation}\label{eq:square} \xymatrix@-1pc{(S \oplus_{d} J, Q\oplus_{d^\flat} J) \ar[r] \ar[d]_-{(p^{\phantom{\flat}}_d, p^\flat_d)} & (S,Q)\ar[d]^-{\underline{0}} \\ (S,Q) \ar[r]^-{(d,d^\flat)} & (S \oplus J[1], Q \oplus J[1]).}
\end{equation}
\end{theorem}

\begin{proof} In order to simplify the exposition, we do not make explicit some of the cofibrant replacements needed to represent maps in homotopy categories by maps in the relevant model categories.  By Proposition~\ref{prop:cruciallemma}, there exists a ring derivation $d\colon S \to S \oplus J[1]$ satisfying the statement of that Proposition. Let $\delta\colon S \to J[1]$ denote its $J[1]$ component.  By Lemma~\ref{lemma:diginto}, we have a homotopy pushout diagram  \[\xymatrix@-1pc{R^{\times} \ar[rr]^-{c_{R}^\gp} \ar[d]_-{\pi^{\times}} && P^\gp \ar[d]^-{(\pi^{\flat})^\gp} \\ S^{\times} \ar[rr]^-{c_{S}^\gp} && Q^\gp.}\] Using its universal property and $\delta \pi = 0$, it follows that
\[0 \colon P^\gp \longrightarrow Q \oplus J[1]\qquad \text{and}\qquad S^{\times} \longrightarrow Q \oplus J[1],\quad v \longmapsto (c_{S}(v), v^{-1}\delta(v)) \]
 define a homomorphism $\epsilon^\flat\colon Q^\gp \to Q \oplus J[1]$ whose precomposition with $Q \to Q^\gp$ will be denoted by $\delta^\flat\colon Q \to Q \oplus J[1]$. Using this map, we define
 \[d^{\flat} = (\id_Q,\delta^\flat)\colon Q \longrightarrow Q \oplus J[1].\] Since $\beta \circ c_{S}$ is the canonical inclusion map $S^{\times} \to S$, one can use the restrictions to $P^\gp$ and $S^{\times}$ in the above homotopy pushout to check that the diagram \[\xymatrix@-1pc{
Q \ar[d]_-{\beta} \ar[r]^-{d^{\flat}} & Q \oplus J[1] \ar[d]^-{\alpha_{J[1]}}\\ 
S \ar[r]^-{d} & S \oplus J[1] 
}\] commutes. Here  $\alpha_{J[1]}\colon Q \oplus J[1] \to S \oplus J[1], \;(s, \xi) \mapsto (\beta(s), \beta(s)\xi)$ denotes the log-structure map. This implies that $(d, d^{\flat})$ is indeed a log-derivation.

 Next we use $(d,d^\flat)$ and the trivial log derivation
 $\underline{0}$ to form the homotopy
 pullback~\eqref{eq:square}. Since the forgetful functors from
 simplicial pre-log rings to simplicial commutative rings and simplicial
 commutative monoids preserve pullbacks, both the ring and the monoid component
 of~\eqref{eq:square} are homotopy pullbacks in the respective
 categories.  By Proposition~\ref{prop:cruciallemma} we get an
 isomorphism \[\xymatrix@-1pc{R \ar[rr]^-{\chi} \ar[dr]_-{\pi} & & S \oplus_{d} J \ar[dl]^{p_d} \\~ & S & ~}\] in
 $\mathrm{Ho}(s\mathcal{A}/S)$. Since $d^{\flat} \circ \pi^{\flat}
 = (\pi^{\flat}, 0)= \underline{0}\circ \pi^{\flat} $, we also get an
 induced map $\chi^{\flat}\colon P \to Q\oplus_{d^\flat} J$ of simplicial monoids
 over $Q$, where $ Q\oplus_{d^\flat} J$
 maps to $Q$ via
 $p^{\flat}_d$. Moreover, since the
 following diagram commutes
\[ \xymatrix@!0@+.5pc{ P \ar[dd]_-{\pi^{\flat}} \ar[rr]^-{\pi^{\flat}} \ar[rd]^-{\alpha} & & Q \ar'[d][dd]^-{d^{\flat}} \ar[rd]^-{\beta} \\
& R \ar[dd]_(.25){\pi} \ar[rr]^(.25){\pi} & & S \ar[dd]^-{d} \\
Q \ar'[r][rr]^(-.5){\underline{0}^{\flat}} \ar[rd]_-{\beta} && Q \oplus J[1] \ar[rd]_-{\alpha_{J[1]}} \\
& S \ar[rr]_-{0} && S\oplus J[1], }
\] 
we get that the pair $(\chi, \chi^{\flat})\colon (R,P) \to (S \oplus_{d} J, Q \oplus_{d^\flat} J)$ is indeed a map in $s\mathcal{L}/(S,Q)$.

We already know from Proposition~\ref{prop:cruciallemma} that $\chi$ is an equivalence of simplicial rings. To prove that $\chi^{\flat}$ is an equivalence, we consider the following commutative diagram of simplicial commutative monoids: \begin{equation}\label{eq:fourfold-square}\xymatrix@-1pc{J \ar[rr]^-{\textrm{exp}} \ar[d] && R^{\times} \ar[rr]^-{c_R} \ar[d]^-{\pi^{\times}} && P \ar[rr] \ar[d]^\pi&& P^\gp \ar[d]^-{\pi^{\flat}} \ar[rr] &&\{1\} \ar[d] \\
 \{1\} \ar[rr] && S^{\times} \ar[rr]^-{c_S} && Q \ar[rr] && Q^\gp \ar[rr]^-{\epsilon^\flat}  && J[1].}
\end{equation}
The outer square is homotopy cartesian and homotopy cocartesian as a
square of simplicial abelian groups. Using
Corollary~\ref{cor:hocart-sAb-sM} and Lemma~\ref{lemma:diginto}, it
follows that the rightmost square in~\eqref{eq:fourfold-square} is
homotopy cocartesian. Since it consists of simplicial abelian groups,
it is also homotopy cartesian in $s\mathcal{AB}$ and hence also in
$s\mathcal M$. The third square (spanned by $\pi$ and $\pi^\flat$) is
cartesian by Lemma~\ref{lemma:strict-exact}. Since it
consists of discrete simplicial monoids, it is also homotopy
cartesian. It follows that the composite  
\[\xymatrix@-1pc{P \ar[r] \ar[d]_-{\pi^{\flat}} & \{1\} \ar[d] \\
Q \ar[r]^-{\delta^{\flat}} & J[1]}\]
of the two rightmost squares is homotopy cartesian.

On the other hand, by definition of $Q\oplus_{d^\flat} J $, the outer square in the diagram 
\[\xymatrix@-1pc{Q\oplus_{d^\flat} J  \ar[d]_{p_d^\flat} \ar[rr] && Q  \ar[d]^{\underline{0}^\flat} \ar[rr] && \{1\} \ar[d]\\
Q \ar[rr]^-{d^\flat} && \ar[rr]^-{\mathrm{proj}} Q\oplus J[1] && J[1]
}\]
is homotopy cartesian. Now it is enough to observe that, since by definition of $\chi^{\flat}$, we have $p^{\flat} \circ \chi^{\flat}= \pi^{\flat}$, and the following diagram 
\[ \xymatrix@!0@+.5pc{ P \ar[dd] \ar[rr]^-{\pi^{\flat}} \ar[rd]^-{\chi^{\flat}} & & Q \ar'[d][dd]^-{\delta^{\flat} } \ar[rd]^-{\textrm{id}} \\
& Q\oplus_{d^\flat} J \ar[dd] \ar[rr]^(.4){p^{\flat}} & & Q \ar[dd]^-{\delta^{\flat}} \\
\{ 1\} \ar'[r][rr] \ar[rd] && J[1] \ar[rd]^-{\textrm{id}} \\
& \{ 1 \} \ar[rr] &&  J[1] }\] 
commutes. The front and rear faces of this diagram are homotopy pullbacks, hence we conclude that $\chi^{\flat}$ is an equivalence of simplicial commutative monoids. 
\end{proof}

\begin{remark} \label{+preciso} Note that the homotopy pullback in Theorem~\ref{thm:strict-sq-zero-log-derivation} is a homotopy pullback in $s\mathcal{L} / (S, Q)$, since the forgetful functor $s\mathcal{L} / (S, Q) \to s\mathcal{L}$ creates homotopy pullbacks. Moreover, Theorem~\ref{thm:strict-sq-zero-log-derivation} holds true if we work in the undercategory $(A, M)/ s\mathcal{L}$, where $(A, M)$ is any simplicial log ring, i.e., $(\pi,\pi^\flat)$ is a map in $(A, M)/ s\mathcal{L}$ and a strict square zero extension of discrete log rings, and the homotopy pullback defining $(S \oplus_{d} J, Q\oplus_{d^\flat} J)$ in Theorem~\ref{thm:strict-sq-zero-log-derivation} is taken in $(A, M)/ s\mathcal{L}$ (or equivalently in $(A, M)/ s\mathcal{L}/ (S, Q)$).
\end{remark}

\section{Derived log-\'etale maps}

Following \cite[3.2]{ka}, we give the following

\begin{definition}
  \label{katologetale}
  A morphism $(f,f^{\flat}) \colon (A,M)\rightarrow (B,N)$ of
  discrete log rings will be called \emph{formally
    log-\'etale} if for any strict square zero extension of discrete integral
  log rings $(g, g^{\flat}) \colon (R,P) \rightarrow (S, Q)$ and
  every commutative diagram
  \[
  \xymatrix@-1pc{
    (S,Q) & \ar[l] (B,N)\\
    (R,P) \ar[u]^{(g,g^{\flat})} & (A,M) \ar[u]_{(f,f^{\flat})}
    \ar[l] }
  \]
  there exists a unique $(h, h^{\flat}) \colon (B,N) \to (R,P)$
  such that the resulting triangles commute.
 
  A morphism $(f,f^{\flat}) \colon (A,M)\rightarrow (B,N)$ of
  discrete log rings will be called \emph{log-\'etale} if
  it is formally log-\'etale and the underlying map $f\colon A
  \rightarrow B$ is finitely presented.\footnote{We here have adopted 
    the convention of Gabber-Ramero \cite{gabber} of not imposing any 
    finiteness condition on the monoid map.}
\end{definition}

\begin{remark}
  \label{katoscomp}
  Note that Kato defines \'etale morphisms only in the category of 
  \emph{fine} log rings, i.e., both $(f,f^{\flat})$ and $(g, g^{\flat})$ 
  are required to be morphisms of fine log rings. Thus if a map of fine 
  log ring $(A,M) \to (B,N)$ is \'etale in the sense of Definition 
  \ref{katologetale}, then it is \'etale in Kato's sense. Since there 
  seems to be no homotopically meaningful way to define the notion of a fine 
   log structure, we have chosen the previous more general 
  definition, i.e, we just probe using square zero extension of integral, but not necessarily fine, log rings.
\end{remark}

For log simplicial rings, we adopt the following
\begin{definition}
  A morphism $(f,f^{\flat}) \colon (A,M)\rightarrow (B,N)$ of
  log simplicial  rings will be called \emph{derived
    formally log-\'etale} if $\mathbb{L}_{(B,N)/(A,M)}$ is
  trivial.

  A morphism $(f,f^{\flat}) \colon (A,M)\rightarrow (B,N)$ of
  log simplicial rings will be called \emph{derived log
    \'etale} if it is formally derived log-\'etale, and the underlying 
  morphism $f$ is homotopically finitely presented \cite[Definition 
  1.2.3.1]{hagII}.
\end{definition}

\begin{proposition}\label{basicpropertiesetale}
The composition of two derived log-\'etale maps is derived log-
  \'etale.  If $f\colon (A,M)\rightarrow (B,N)$ and $g\colon 
  (A,M)\rightarrow (C,O)$ are maps of simplicial pre-log algebras, and 
  $f$ is derived log-\'etale, then the homotopy base change map 
  \[(C,O)\longrightarrow (B,N) \textstyle\coprod^{h}_{(A,M)} (C,O)\] is 
  derived log-\'etale.
\end{proposition}

\begin{proof} First of all, observe that being of finite presentation is stable under composition and base-change. The remaining statements about the cotangent complex, follow from the transitivity sequence, and from the so-called flat base-change, i.e., from Proposition~\ref{basictriangles}.
\end{proof}

Since a map of log simplicial rings $(f,f^{\flat})\colon (A,M) \to (B,N)$ 
is strict if and only if the induced map of pre-log rings $(B,f_*(M)) 
\to (B,N)$ is a weak equivalence in the log model structure, it follows 
that strict maps are preserved under base change:

\begin{lemma}
  \label{lem:StrictBaseChange}
  Let
  \[\xymatrix@-1pc{ (A,M) \ar[d] \ar[r] & (B,N) \ar[d] \\ (R,P) \ar[r]
    & (S,Q)}\] be a homotopy pushout of log simplicial rings. If
  $(A,M) \to (B,N)$ is strict, then so is $(R,P) \to (S,Q)$.
\end{lemma}

\begin{proof}
  We consider the induced square
  \[\xymatrix@-1pc{ (A,M) \ar[d] \ar[r] &  (B,f_*(M)) \ar[d] \ar[r] &(B,N) \ar[d] \\ (R,P) \ar[r]
    & (S,f_*(P)) \ar[r] & (S,Q).}\] Here the outer square is homotopy
  cocartesian by assumption, and it follows easily that the left hand
  square is homotopy cocartesian. Hence the right hand square is
  homotopy cocartesian. If $(B,f_*(M)) \to (B,N)$ is a weak equivalence
  in the log model structure, this implies that $(S,f_*(P)) \to (S,Q)$
  also has this property.
\end{proof}

The following Theorem shows that the previous notion of log-\'etaleness 
implies the classical one on the truncation. We recall the following 
notation: Given a model category $C$ and objects $A$ and $B$ of $C$, we 
will denote the respective slice model categories by $A/C$, $C/B$ and 
$A/C/B$.

\begin{theorem}
  \label{etaleontruncations}
  If a morphism $(f,f^{\flat}) \colon (A,M)\rightarrow (B,N)$ of log
  simplicial rings is derived log-\'etale and $\pi_0f^{\flat}$ is of
  finite presentation, then the induced morphism 
  $(\pi_{0}f, \pi_{0}f^{\flat} ) \colon (\pi_{0} A, \pi_{0} M)
  \rightarrow (\pi_{0}B, \pi_{0} N)$ is a log-\'etale morphism of discrete log
  rings (in the sense of Definition~\ref{katologetale}).
\end{theorem}

\begin{proof} First of all observe that the left Quillen functor
  $\pi_{0}: s\mathcal{P} \rightarrow \mathcal{P}$ preserves finitely
  presented objects. So we are left to prove that $(\pi_{0}f,
  \pi_{0}f^{\flat} )$ is formally \'etale. Let $(\pi,\pi^\flat)\colon (R,P)
  \rightarrow (S, Q)$ be a strict square zero extension of discrete integral
  log rings under $(A,M)$, with square zero ideal $J$.

We have to prove that the canonical map \[\Hom_{(\pi_0 A,\pi_0 M )/\mathcal{P}}((\pi_0B,\pi_0N), (R,P))\longrightarrow \Hom_{(\pi_0 A,\pi_0 M )/\mathcal{P}}((\pi_0B,\pi_0N), (S,Q))\] is bijective. Since $(R,P)$ and $(S,Q)$ are discrete, this is equivalent to showing that the canonical map \[u\colon  \Map_{(A,M)/s\mathcal L}((B,N), (R,P))\longrightarrow \Map_{(A,M)/s\mathcal L}((B,N), (S,Q))\] is a weak equivalence of simplicial sets. This is true if and only if for any $0$-simplex  $\underline{\varphi} = (\varphi,\varphi^\flat)$ in $\Map_{(A,M)/s\mathcal L}((B,N), (S,Q))$, the homotopy fiber \[\mathrm{hofib}(\underline{\varphi}) = \mathrm{hofib}(u; \underline{\varphi})\] is non-empty and contractible. In order to establish this, let us write $B_{\underline{\varphi}}$ for $(B,N)$ viewed as an object in $(A,M)/s\mathcal{L}/ (S,Q)$ via $\underline{\varphi}$, and let $\varphi^* J[1]$ be $J[1]$ viewed as a $B$-module via $\varphi$. We consider the map \[\rho_{\underline{\varphi}}\colon  \Map_{(A,M)/s\mathcal L/(S,Q)} (B_{\underline{\varphi}}, (S,Q) \oplus J[1]) \longrightarrow \Map_{(A,M)/s\mathcal L/(S,Q)} ((B,N), (B,N)\oplus \varphi^*J[1])\] induced by sending \[(D, D^{\flat})\colon B_{\underline{\varphi}}\rightarrow (S,Q)\oplus J[1]\] to the map $(B,N)\rightarrow (B,N)\oplus \varphi^{*}J[1]$ whose projection to $(B,N)$ is the identity, and whose projection to $\varphi^*J[1]$ is given, on $B$ and $N$, respectively, by the composites \[\xymatrix{B \ar[r]^-{D} &S\oplus J[1] \ar[r]^-{\mathrm{pr}_{J[1]}} & J[1]} \qquad\text{and}\qquad\xymatrix{N  \ar[r]^-{D^{\flat}} & Q\oplus J[1] \ar[r]^-{\mathrm{pr}_{J[1]}} & J[1]}.\]
In the notation of Theorem~\ref{thm:strict-sq-zero-log-derivation}, we obtain a diagram    \[\xymatrix{\mathrm{hofib}(\underline{\varphi}) \ar[d] \ar[r] & \bullet \ar[d]^-{\underline{\varphi}} \\
\Map_{(A,M)/s\mathcal L/(S,Q)}(B_{\underline{\varphi}}, (R,P)) \ar[d] \ar[r]^-{p_{d}} & \Map_{(A,M)/s\mathcal L/(S,Q)}(B_{\underline{\varphi}}, (S,Q)) \ar[d]^-{(d,d^\flat)} \\
\Map_{(A,M)/s\mathcal L/(S,Q)}(B_{\underline{\varphi}}, (S,Q)) \ar[d] \ar[r]^-{\underline 0} & \Map_{(A,M)/s\mathcal L/(S,Q)}(B_{\underline{\varphi}}, (S,Q)\oplus J[1]) \ar[d]^-{\rho_{\underline{\varphi}}} \\
\bullet \ar[r]_-{\underline 0} &  \Map_{(A,M)/s\mathcal L/(B,N)} ((B,N), (B,N)\oplus \varphi^*J[1]). }\]
The top square is homotopy cartesian by definition. The middle one is homotopy cartesian by Theorem~\ref{thm:strict-sq-zero-log-derivation}, the fact that $\Map$ preserves homotopy cartesian squares, and the fact that the forgetful functor $(A,M)/s\mathcal{L}/ (S,Q) \to (A,M)/s\mathcal{L}$ creates homotopy pullbacks, so that the homotopy pullback of Theorem~\ref{thm:strict-sq-zero-log-derivation} is actually a homotopy pullback in $(A,M)/s\mathcal{L}/(S,Q)$ (see Remark~\ref{+preciso}). The bottom square is homotopy cartesian by definition of the map $\rho_{\underline{\varphi}}$.
If we set $D_{\underline{\varphi}}= \rho_{\underline{\varphi}} \circ (d,d^\flat) \circ \underline{\varphi}$, we see that $\mathrm{hofib}(\underline{\varphi})$ is non-empty if and only if $D_{\underline{\varphi}}$ is zero in \[\pi_{0} \, \Map_{(A,M)/s\mathcal L/(B,N)}((B,N), (B,N)\oplus \varphi^*J[1]) \cong \Ext_{B}^1( \mathbb{L}_{(B,N)/(A,M)}, \varphi^* J),\] and in this case \[\begin{split}\mathrm{hofib}(\underline{\varphi})&\cong \Omega_{0}\Map_{(A,M)/s\mathcal L/(B,N)} ((B,N), (B,N)\oplus \varphi^* J[1])\\ &\cong \Omega_{0} \Map_{B-\Mod}(\mathbb{L}_{(B,N)/(A,M)},\varphi^*J[1]).
\end{split}
\] If $\mathbb{L}_{(B,N)/(A,M)}\cong 0$, $\mathrm{hofib}(\underline{\varphi})$ is then non-empty and contractible.
\end{proof}

\section{Derived log-smooth maps}

Following \cite{ka}, we give the following

\begin{definition}\label{katologsmooth}
A morphism $(f,f^{\beta}) \colon (A,M) \rightarrow (B,N)$ of 
    discrete log rings, will be called \emph{formally 
      log-smooth}  if for any strict square zero extension of discrete integral
    log rings $(g,g^{\flat}) \colon (R,P) \rightarrow (S,Q)$, the 
    canonical map \[\Hom_{\mathcal{L}_{(A,M)}}((B,N), 
    (R,P))\longrightarrow \Hom_{\mathcal{L}_{(A,M)}}((B,N), (S,Q)\] is 
  surjective. Here $\mathcal{L}_{(A,M)}$ denotes the slice category of 
  $\mathcal{L}$ over $(A,M)$.

A morphism $(f,f^{\flat})\colon (A,M)\rightarrow (B,N)$ of 
    discrete log rings, will be called \emph{log-smooth}  if 
    it is formally log-smooth, and the underlying map $f\colon A 
    \rightarrow B$ is finitely presented as a map of commutative 
    algebras.
\end{definition}

\begin{remark}
  \label{katoscompsmooth}
  Remark~\ref{katoscomp} applies analogously to log smooth maps: We
  define smoothness on a larger class of maps than Kato does, and if a
  map of fine log rings is smooth in the sense of
  Definition~\ref{katologsmooth}, then it is smooth in Kato's sense.
  % As observed in Remark \ref{katoscomp} for log \'etale maps, Kato defines mooth morphisms only in the category of 
  % \emph{fine} log rings, i.e., both $(f,f^{\flat})$ and $(g, g^{\flat})$ 
  % are required to be morphisms of fine log rings. Thus if a map of fine 
  % log rings $(A,M) \to (B,N)$ is smooth in the sense of Definition 
  % \ref{katologsmooth}, then it is smoothin Kato's sense. Since there 
  % seems to be no homotopically meaningful way to define the notion of a fine 
  %  log structure, we have chosen the previous more general 
  % definition, i.e we just probe using square zero extension of integral, but not necessarily fine, log rings.
\end{remark}

\begin{definition}
A morphism $(f, f^{\flat})\colon (A,M)\rightarrow (B,N)$ of log
    simplicial  rings will be called \emph{derived formally log-smooth}  
    if \[\Hom_{\mathrm{Ho(B-\textrm{Mod})}}(\mathbb{L}_{(B,N)/(A,M)}, J) 
    \simeq 0\] for any simplicial $B$-module $J$ with $\pi_{0} \, J =0$.  

A morphism $(f, f^{\flat}) \colon (A,M)\rightarrow (B,N)$ of log
    simplicial rings will be called \emph{derived log-smooth}  if it is 
    formally log-smooth and $f \colon A \to B$
    is homotopically finitely presented.
\end{definition}

The following Theorem shows that the notion of derived log-smoothness 
implies the classical one on the truncations.

\begin{theorem}
  If $(f,f^{\flat}) \colon (A,M)\rightarrow (B,N)$ of log simplicial
  rings is derived log-smooth, then the induced morphism  $(\pi_{0}f,\pi_0f^\flat) \colon (\pi_{0}A, \pi_{0}M)
  \rightarrow (\pi_{0}B, \pi_{0}N)$ is a log-smooth morphism of discrete log rings (in the sense of Definition \ref{katologsmooth}).
\end{theorem}

\begin{proof} Let $(\pi,\pi^\flat)\colon (R,P) \rightarrow (S, Q)$ be a 
  strict square zero extension of discrete integral log rings under $(A,M)$, with square zero ideal $J$. Following the arguments and the notation in the proof of Theorem~\ref{etaleontruncations} shows that the map 
 \[\Hom_{(\pi_0 A,\pi_0 M )/\mathcal{L}}((\pi_0B,\pi_0N), (R,P))\longrightarrow \Hom_{(\pi_0 A,\pi_0 M )/\mathcal{L}}((\pi_0B,\pi_0N), (S,Q))\]
is surjective if and only if for any $0$-simplex  
  $\underline{\varphi}$ in  $\Map_{(A,M)/s\mathcal L}((B,N), (S,Q))$, the associated element  $D_{\underline{\varphi}}$ is zero in 
   \[\begin{split}&\pi_{0} \, \Map_{(A,M)/s\mathcal L/(B,N)}((B,N), (B,N)\oplus \varphi^*J[1])\\ \cong \,&\Ext_{B}^1( \mathbb{L}_{(B,N)/(A,M)}, \varphi^* J) 
     \cong \Hom_{\mathrm{Ho(B-\Mod)}}(\mathbb{L}_{(B,N)/(A,M)},
     \varphi^*J[1]).
\end{split}
\]
 But $\pi_{0}(\varphi^*J[1])=0$, so the result follows.
\end{proof}

\section{Derived log stacks}
In giving our definitions, we will not mention explicitly the proper choices of universes: the reader will find they are the same as in \cite{hagI}.

\subsection{Derived log prestacks}

Throughout we fix a base commutative ring $k$. If we view $k$ as a constant
simplicial ring with the trivial simplicial pre-log structure, then
the category of pre-log simplicial $k$-algebras is the category
pre-log simplicial rings under $k$. It is denoted by $s\mathcal P_k$
and inherits an injective and a projective model structure from
$s\mathcal P$. Likewise, we obtain a model category of log simplicial
rings $s\mathcal{L}_{k}$ from Theorem~\ref{thm:log-model-str} as the
comma category $k\downarrow s\mathcal L$.
\begin{definition} 
The category of \emph{derived log affines over $k$} is the opposite category $\dLogAff_{k}$ of $s\mathcal{L}_{k}$, 
and we let 
\[ \mathrm{SPr}(\dLogAff_{k}) := \mathcal S^{\textrm{dLogAff}_{k}^{\textrm{op}}} = \mathcal S^{s\mathcal{L}_{k}}
\] be the category of simplicial presheaves on derived log affines over $k$.
\end{definition}

Note that $\dLogAff_{k}$ is a simplicial model category, and that
 $\mathrm{SPr}(\dLogAff_{k})$ is simplicially enriched by 
 \[\Hom_{\mathrm{SPr}(\textrm{dLogAff}_{k})}(F,G)_{n}:= 
 \Hom_{\mathrm{SPr}(\textrm{dLogAff}_{k})}(F\times \Delta^{n},G)\] where 
 \[(F\times \Delta^{n}) (A,M):= F(A,M)\times \Delta^{n}.\]

\begin{proposition}
  The category $\mathrm{SPr}(\dLogAff_{k})$ admits a left
  proper cellular model structure where the weak
  equivalences and the fibrations are defined
  object-wise. 
\end{proposition}

\begin{proof} This is \cite[Propositions A.1.3(1) and A.2.5]{hagI}.
\end{proof}

Consider the Yoneda functor $$\dLogAff_{k} \longrightarrow \mathrm{SPr}(\textrm{dLogAff}_{k}),\,\,\,\, X\longmapsto h_{X}:= \Hom_{\textrm{dLogAff}_{k}}(-, X),$$ and define $$h_{W}:= \{ h_{w}:h_X \rightarrow h_Y \,|\, w:X\rightarrow Y \, \textrm{a weak equivalence in \,} \textrm{dLogAff}_{k} \}.$$

\begin{definition} The category of \emph{log prestacks over $k$} is the model category $\dLogAff^{\wedge}_{k}$ obtained as the left Bousfield localization of  $\mathrm{SPr}(\textrm{dLogAff}_{k})$ with respect to $h_{W}$.
\end{definition}

\begin{remark}
In the notations of \cite[Definitions 2.3.3 and 4.1.4]{hagI}, the model category $\dLogAff^{\wedge}_{k}$ (with the appropriate choice of universes) is denoted as $(\textrm{dLogAff}_{k}, S)^{\wedge}$, where $S$ stands for the weak equivalences in $\textrm{dLogAff}_{k}$.
\end{remark}

Note that, by standard properties of left Bousfield localizations (see e.g.~\cite{Hirschhorn_model}), $\mathrm{Ho}(\dLogAff^{\wedge}_{k})$ can be identified with the full subcategory of $\mathrm{Ho}(\mathrm{SPr}(\textrm{dLogAff}_{k}))$ consisting of functors $F\colon \textrm{dLogAff}_{k}^{\textrm{op}} \longrightarrow \mathcal{S}$ preserving weak equivalences.

We are now able to define a derived log analog of the spectrum functor.
 
\begin{definition} \label{Spec}
We define the \emph{derived log spectrum} functor $\textrm{Spec}$ as 
follows $$\textrm{Spec}\colon \textrm{Ho}(\dLogAff_{k}) \longrightarrow 
\textrm{Ho}(\textrm{dLogAff}^{\wedge}_{k})\, ,\,\,\,\, (A,M) \longmapsto 
\bfHom_{s\mathcal{P}}(Q(A,M), R(-))$$ where $Q(-)$ (respectively, 
$R(-)$) denotes a cofibrant (resp., fibrant) replacement functor in the 
model category $s\mathcal{P}^{\textrm{proj}}$, and 
$\bfHom_{s\mathcal{P}}(-,-)$ the simplicial enrichment in 
$s\mathcal{P}$.
\end{definition}

Equivalently, we could have defined $\textrm{Spec}$ as in 
\cite[Definition 4.2.5]{hagI}. The model category version of the Yoneda 
lemma (\cite[Corollary 4.2.4]{hagI}), tells us that

\begin{proposition}
The $\textrm{Spec}$ functor is fully faithful, and for any $(A,M)\in 
s\mathcal{P}$, and any $F \in \dLogAff^{\wedge}_{k}$, we have a 
canonical isomorphism in $\textrm{Ho}(\mathcal{S})$, 
$$\Map_{\textrm{dLogAff}^{\wedge}_{k}}(\mathrm{Spec} (A,M), F) \simeq 
F(A,M).$$
\end{proposition}

\subsection{Derived log stacks}

\begin{definition} A family $\{(A,M) \rightarrow (A_i , M_{i}) \}_{i\in 
    \textrm{I}}$ of morphisms in $s\mathcal{P}_{k}$ is called a 
  \emph{strict log-\'etale covering family} of $(A,M)$ in $\dLogAff_{k}$ 
  if \begin{itemize}
\item each $(A,M) \longrightarrow (A_i , M_{i})$ is a strict log-\'etale 
  morphism (of simplicial pre-log $k$-algebras), and \item there exists 
  a finite subset $\textrm{J} \subseteq \textrm{I}$ such that the family 
  of base-change functors $\{ -\otimes_{A}^{\mathbb{L}} A_{j}\colon 
    \mathrm{Ho}(s\textrm{Mod}_{A}) \longrightarrow 
    \mathrm{Ho}(s\textrm{Mod}_{A_{j}})\}_{j\in \textrm{J}}$ is 
  conservative.
\end{itemize}
\end{definition}

\begin{proposition}
The collection of strict log-\'etale covering families form a model 
pre-topology on the model category $\dLogAff_{k}$ in the sense 
of \cite[Definition 4.3.1]{hagI}.
\end{proposition}

\begin{proof} This follows immediately from stability of strict 
  log-\'etale maps with respect to composition and homotopy pullbacks 
  (Proposition  \ref{basicpropertiesetale} and 
  Lemma~\ref{lem:StrictBaseChange}).
\end{proof}

\begin{definition} We denote by \emph{str-log-\'et} both the model 
  pre-topology, given by strict log-\'etale covering families, on 
  $\dLogAff_{k}$, and the Grothendieck topology on 
  $\textrm{Ho}(\dLogAff_{k})$ generated by the induced 
  pre-topology.
\end{definition}

To any $F\in \mathrm{SPr}(\dLogAff_{k})$, we can associate the 
\emph{sheaf} of connected components $\pi_{0}(F)$ on the strict log 
\'etale (usual) site $(\textrm{Ho}(\dLogAff_{k}), 
\textrm{str-log-\'et})$.  And, for any $i>0$, any fibrant $X\in 
\dLogAff_{k}$, and any $s\in F(X)_{0}$, we can consider the 
\emph{sheaf} $\pi_{i}(F, s)$ on the comma site 
$(\textrm{Ho}(\dLogAff_{k}/X), \textrm{str-log-\'et})$ 
(\cite[Definition  4.5.3.]{hagI}).

\begin{definition} A map $f:F\longrightarrow G$ in $ \mathrm{SPr}(\dLogAff_{k})$ is called a $\pi_{*}$\emph{-isomorphism} if the induced maps of sheaves $$\pi_0 (F) \longrightarrow \pi_0 (G), $$ $$\pi_{i}(F,s) \longrightarrow \pi_{i}(G, f(s))$$ are isomorphisms, for any $i >0$, any fibrant $X$, and any $s\in F(X)_{0}$.
\end{definition}

\begin{theorem} \label{etalelogstacks} There is a model structure on $\mathrm{SPr}(\dLogAff_{k})$ in which the cofibrations are the same as those in $\textrm{dLogAff}^{\wedge}_{k}$, and the weak equivalences are $\pi_{*}$-isomorphisms.
\end{theorem}

\begin{proof} This follows from \cite[Theorem 4.6.1]{hagI}.
\end{proof}

\begin{definition} 
%We denote by $\textbf{dLogSt}_{k}$ the model category structure on $\mathrm{SPr}(\dLogAff_{k})$ given by Theorem \ref{etalelogstacks}.  % We should only introduce the notation if we use it. 
%It will be called the \emph{model category of derived log stacks}, and its homotopy category will be simply denoted by $\textrm{dLogSt}_{k}$.
The model category structure on $\mathrm{SPr}(\dLogAff_{k})$ given by Theorem \ref{etalelogstacks} will
be called the \emph{model category of derived log stacks}, and its homotopy category will be simply denoted by $\textrm{dLogSt}_{k}$
\end{definition}

It follows from the proof of Theorem \ref{etalelogstacks}, and from basic properties of left Bousfield localizations, that $\textrm{dLogSt}_{k}$ can be identified with the full subcategory of $\mathrm{Ho}(\mathrm{SPr}(\dLogAff_{k}))$ consisting of functors $F\colon \textrm{dLogAff}_{k}^{\textrm{op}} \longrightarrow \mathcal{S}$ such that
$F$ preserves weak equivalences and 
$F$ satisfies \emph{strict log-\'etale hyperdescent}, i.e., the 
  canonical map $$F(X) \longrightarrow  
  \mathrm{holim}_{\Delta^{op}}F(H_{\bullet}):=
  \mathrm{holim}_{\Delta^{op}} 
  \Map_{\dLogAff^{\wedge}_{k}}(H_{\bullet}, F) $$ is an 
  isomorphism in $\mathrm{Ho}(\mathcal{S})$, for any strict log-\'etale 
  pseudo-representable hypercover $H_{\bullet} \rightarrow h_{X}$ of $X$ 
  (see \cite[Definition 4.6.5]{hagI}).

In particular, we will say that an object $F \in 
\textrm{Ho}(\dLogAff^{\wedge}_{k})$ \emph{is} a derived log 
stack, if it satisfies the strict log-\'etale hyperdescent condition.

\begin{proposition}\label{subcanonical}
The strict log-\'etale model pre-topology on the model category 
$\dLogAff_{k}$  is sub-canonical (\cite[Definition 1.3.1.3]{hagII}), 
i.e., $\textrm{Spec} (A,M)$ is a derived log stack, for any $(A,M) 
\in s\mathcal{P}$.
\end{proposition}

\begin{proof} We will only prove the case of a strict log-\'etale 
  representable hypercover, leaving to the reader the general case of a 
  a strict log-\'etale pseudo-representable hypercover (as in the proofs of Lemma 2.2.2.13 and Lemma 1.3.2.3 (2) in \cite{hagII}). By using finite 
  products, we can assume that we are working with a strict log-\'etale 
  covering family given by a single map $(A,M) 
  \rightarrow (B,N)$ in $s\mathcal{P}$.
 We have to show that the morphism $(A,M) \to |(B,N)_{\bullet}|$ is an 
 isomorphism in $\mathrm{Ho}(s\mathcal{L}_k)$. Let 
 $\mathrm{Ho}(s\mathcal{L}_k)^{\str}$ denote the sub-category of 
 $\mathrm{Ho}(s\mathcal{L}_{k})$ spanned by log simplicial rings with 
 strict morphisms. Since strictness is preserved under homotopy 
 colimits, $(A,M) \to |(B,N)_{\bullet}|$ gives a morphism in 
 $\mathrm{Ho}(s\mathcal{L}_k)^{\str}$. Let $U \colon s\mathcal{L}_k \to 
 s\mathrm{Alg}_k$ denote the functor that forgets the log structure. By 
 strictness, the induced functor $U \colon 
 \mathrm{Ho}(s\mathcal{L}_k)^{\str} \to \mathrm{Ho}(s\mathrm{Alg}_k) $ 
 is conservative. The claim then follows from the string of isomorphisms 
 in $\mathrm{Ho}(s\mathrm{Alg}_k)$
    \[
      U(A,M) \to |U(B,N)_{\bullet}| \to U(|(B,N)|_{\bullet})
    \]
    where the first isomorphism comes from descent for the \'etale 
    topology on $\mathrm{dAff}_k$, and the second isomorphism holds because $U$ commutes 
    with homotopy colimits.
\end{proof}

By Proposition \ref{subcanonical}, the $\textrm{Spec}$ functor factors 
as a fully faithful functor $$\textrm{Spec}\colon 
\textrm{Ho}(\dLogAff_{k}) \longrightarrow \textrm{dLogSt}_{k}.$$

\begin{remark}
One might also consider the not necessarily strict log-\'etale model pre-topology on the model category 
$\dLogAff_{k}$. The problem with this model topology is that it is very likely that it is not subcanonical. This is closely related to the fact that
the log-\'etale topology on general (i.e not necessarily fs) log schemes is probably also not subcanonical.
\end{remark}

\subsection{Geometric derived log stacks}

By following the same path as in \cite{hagII}, we give the following inductive definition

\begin{definition} 
A derived log stack is \emph{$(-1)$-geometric} if it is representable, i.e., isomorphic in $\textrm{dLogSt}_{k}$ to $\textrm{Spec}(A,M)$ for some simplicial pre-log $k$-algebra $(A,M)$.
Let $n\geq 0$ be an integer. 
\begin{itemize}
\item A derived log stack $F\in \textrm{dLogSt}_{k}$ is \emph{$n$-geometric} if 
\begin{itemize}
\item the diagonal map $F \longrightarrow F \times F$ is $(n-1)$-representable
\item There exists a family $\{\mathrm{Spec}(A_{i},M_{i})\}_{i\in 
    \textrm I}$ of representable derived stacks, and a morphism 
  $$p\colon \textstyle\coprod_{i}\textrm{Spec}(A_{i},M_{i}) 
  \longrightarrow F,$$ called an \emph{atlas} for $F$, such that
\begin{itemize}
\item the sheafification of $\pi_{0}(p)$ is an epimorphism of sheaves of 
  sets on the site $(\textrm{dLogSt}_{k}, \textrm{str-log-\'et})$; \item 
  the induced morphism $p_i\colon \textrm{Spec}(A_{i},M_{i}) 
  \longrightarrow F$ is log-smooth, for any $i\in \textrm{I}$.
\end{itemize}
\end{itemize}
\item A morphism $f :F \longrightarrow G$ in $\textrm{dLogSt}_{k}$ is \emph{$n$-representable} if for any representable $X$
and any morphism $X \longrightarrow G$, the derived log stack $F\times_{G} X$ is $n$-geometric.
\item An $n$-representable morphism $f :F \longrightarrow G$ in $\textrm{dLogSt}_{k}$ is \emph{log-smooth} if for any representable $X$
and any morphism $X \longrightarrow G$, there exists an atlas $\coprod_{i}Y_{i} \longrightarrow F\times_{G} X$ for $F\times_{G} X$ such that each induced map $Y_i \longrightarrow X$ is log-smooth between representable derived stacks.
\end{itemize}

\end{definition}

The statement of the Artin property for derived log stacks, and the corresponding version of Lurie's representability criterion will be treated in a sequel to this paper.

\begin{remark}(Pre log and log modules.)
If $(A,M)$ be a simplicial pre-log algebra, there is an 
obvious category $\mathbf{PreLogMod}_{(A,M)}$ of pre-log modules 
over $(A,M)$,
whose objects are triples $(S,P, \varphi\colon S\to P )$ where $S$ is a simplicial $M$-module (i.e., a simplicial set endowed with an action of the simplicial monoid $M$), $P$ is a simplicial $A$-module  , and $\varphi$ is a map of simplicial sets that is equivariant with respect to the structure map $\alpha\colon M \to A$, i.e., such that the following diagram commutes \[\xymatrix@-1pc{M \times S \ar[r] \ar[d]_-{\alpha \times \varphi} & M \ar[d]^-{\varphi} \\ A \times P \ar[r] & P }\] and whose morphisms are the natural ones. There is a model structure on $\mathbf{Mod}_{(A,M)}$ where weak equivalences (resp. fibrations) are pairs $(f,g)$ where $f$ is a weak equivalence (resp. a fibration) of simplicial sets, and $g$ is a weak equivalence (resp. a fibration) of simplicial $A$-modules.  Direct and inverse image functors define a Quillen pair, and there is a natural monoidal structure on $\mathbf{PreLogMod}_{(A,M)}$ such that algebras in $\mathbf{Mod}_{(k,1)}$ are exactly pre-log $k$-algebras. However, $\mathbf{PreLogMod}_{(A,M)}$ is very much non additive. This is reflected by the fact that we have functors $$\mathbf{AbGrps}((k,1)/\mathcal{P}/(A,M)) \hookrightarrow \mathbf{AbMonoids}((k,1)/\mathcal{P}/(A,M)) \longrightarrow \mathbf{Mod}_{(A,M)}$$ where the left one is not an equivalence (while it is in the non pre-log case) and the right one is not essentially surjective (while it is an equivalence in the non pre-log case). One might however use \cite{marty}  to define a notion of flat topology on pre-log algebras (viewed as algebras in $\mathbf{Mod}_{(k,1)}$). Unfortunately, these flat maps have flat underlying maps of schemes, so they are not very interesting.

When $(A,M)$ is a simplicial pre-log algebra with structure map $\alpha$, 
there is a log variant $\mathbf{LogMod}_{(A,M)}$ of 
$\mathbf{PreLogMod}_{(A,M)}$, where we only consider those pre 
log modules $(S,P, \varphi\colon S\to P)$ such that  the map 
$\alpha^{-1}(A_{P}^{\times}) \to A_{P}^{\times}$ is a weak equivalence 
(here $A^{\times}_P$ denotes the connected components of $A$ acting as 
equivalences on $P$). We have not fully investigated the homotopy and 
monoidal structures on this category.

From a general point of view, in order to get an alternative theory of derived log geometry along these lines, we think it might be interesting to proceed as follows. Embed the category of (pre) log rings in the category of arrows between commutative monoids. This embedding is not full so something new is obtained. Then we may use the approach sketched in~\cite[\S 5.3]{TV-specZ} and~\cite{marty} to build a Zariski, flat or smooth topology for arrows between $\mathbb{S}_1$-derived schemes (i.e., the geometric objects of derived geometry over the monoidal model category of simplicial sets), and explore the derived geometry of objects arising via gluing (pre) log rings. This would roughly correspond classically to partially disregard the fact that there is an underlying scheme of a log scheme. This work remains to be done, and we feel like it is a worthwhile task since it might yield a new insight in the foundations of classical log geometry, too.

\end{remark}

\section{An example}

This section provides an example of a non-trivial derived log stack. We 
construct a derived version of the logarithmic moduli of stable maps 
introduced by Gross and Siebert.

We begin by producing an inclusion functor from the category of stacks 
over discrete log rings to the category of derived stacks over log 
simplicial rings. To accomplish this, we endow the category of discrete 
log rings with the trivial model structure. Then the inclusion functor
\[
  i \colon \sL_k \to s\sL_k
\]
from the category of log rings under a base ring $k$ to the category of 
log simplicial rings under a base ring $k$ is a right Quillen functor.  
As a consequence we obtain a Quillen adjunction for the categories of 
pre-stacks
\[
  i_{!} \colon \mathcal{S}^{\sL_k} \rightleftarrows
  \mathrm{dLogAff}_k^{\wedge}
  \colon i^*.
\]
Here $\mathcal{S}^{\sL_k}$ is the category of simplicial pre-sheaves on 
$\sL_k$ equipped with the projective model structure.

We equip the category $\sL_k$ with the strict \'etale topology, and 
using the same construction as in Theorem \ref{etalelogstacks} we can 
define the model category of higher log stacks (see \cite[Section 
2.1]{hagII} for the construction of (non-derived) higher stacks in the 
non-logarithmic context).

To verify that the above adjunction descends to the category of stacks 
we have to check that $i$ preserves coproducts, equivalences and 
hypercovers.  The only non-trivial part is to verify that $i$ preserves 
strict \'etale morphisms, since in the discrete case \'etaleness is 
characterized by the vanishing of the one-truncated cotangent complex, 
whereas in the non-discrete case the full cotangent complex must vanish.

\begin{lemma}
  Let $(A,M) \to (B,N)$ be a strict \'etale morphism of discrete 
  log rings. Then $\IL_{(B,N)/(A,M)} \simeq 0$.
\end{lemma}
\begin{proof}
  Since the morphism is strict, we have an equivalence 
  $\IL_{(B,N)/(A,M)} \simeq \IL_{B/A}$, so that we deduce that 
  $\tau_{\leq 1} \IL_{B/A} \simeq 0$. But by \cite[Prop.  
  3.1.1]{illusie} this implies that $\IL_{B/A} \simeq 0$. We conclude 
  using again the equivalence $\IL_{(B,N)/(A,M)} \simeq \IL_{B/A}$.
\end{proof}

In practice, one usually deals with stacks not defined over the entire 
category of log rings, but only the category of fine and saturated log 
rings. Such a log stack over the category of fine and saturated log 
rings is usually defined as a category fibred in groupoids over this 
category. Using the inclusion functor from groupoids to simplicial sets 
and the Grothendieck construction, we can view every such category 
fibred in groupoids as a simplicial set valued functor on the category 
of fine and saturated log rings. The main example we have in mind is the 
following.
\begin{example} \cite[Def. 1.3]{siebert}
  \label{ex:Mgn}
  Assume our base is a separably closed field $k$.  Denote by 
  $M_{g,n}^{\log, \pre}$ the functor that assigns to every fine and 
  saturated log ring $(A,M)$ the groupoid of proper log-smooth and 
  integral morphisms $f \colon (C, \mathcal{M}) \to \Spec (A,M)$ 
  together with $n$ sections $s_i \colon \Spec (A,M) \to (C, 
  \mathcal{M})$ such that every fibre of $f$ is a reduced and connected 
  curve of genus $g$, and if $U \subset C$ is the non-critical locus of 
  $f$, then $\mathcal{M} | _U \simeq f^* M \oplus \bigoplus_i (s_i)_* 
  \IN_A $.
\end{example}

\begin{remark}
  Note that since we are in the relative situation over a separably 
  closed field $k$ and since in the above examples the log schemes are 
  assumed to be fine and saturated, this ensures that the geometric 
  fibres have at worst nodal singularities by \cite[Theorem 1.3]{fkato}.
\end{remark}

If we now let $\mathcal{L}_k^{\fs}$ denote the category of fine and 
saturated log rings, we then have an inclusion $j \colon 
\mathcal{L}_k^{\fs} \to \mathcal{L}_k$. Arguing as above, we obtain an 
adjunction
\[
  j_{!} \colon \mathcal{S}^{\sL _k^{\fs}} \rightleftarrows 
  \mathcal{S}^{\sL_k}
  \colon j^*
\]
between the categories of simplicial pre-sheaves equipped with the 
projective model structures, and this again descends to the categories 
of stacks with respect to the strict \'etale topology.

Using the composition $i_! \circ j_!$ we can regard any category fibred 
in groupoids over the category of fine and saturated log schemes as a 
derived log stack. By combining this composition and Example 
\ref{ex:Mgn} we can construct the derived moduli of stable maps over a 
fine and saturated base log $k$-scheme $(S, \mathcal{M}_S)$.

\begin{definition}
  Let $(S, \mathcal{M})$ be a fine and saturated log scheme over a 
  separably closed field $k$, and denote by $
  \mathrm{dLogSt}_{(S,\mathcal{M})}$ the comma category of 
  $\mathrm{dLogSt}_k$ over $(S,\mathcal{M})$.  Let $C$ denote the 
  universal curve over $M_{g,n}^{\log,\pre}$, and let $X$ be a derived 
  affine log scheme over $(ij)_{!}S$. We then defined the derived moduli 
  of stable maps as
  \[
    M(X) =  \Map_{\mathrm{dLogSt}_{(S,\mathcal{M})}/(ij)_!  
      M_{g,n}^{\log, \pre}} \left( (ij)_{!}C, X \times 
      (ij)_{!}M_{g,n}^{\log,\pre} \right)
  \]
\end{definition}

Note that we have not proven that $M(X)$ is algebraic. We hope to return 
to this in a future paper. If an Artin-Lurie type representability 
theorem \cites{artin, DAGXIV} for derived log stacks were available, 
this would be an immediate consequence. Once algebraicity is proven one 
can compute the cotangent complex of the derived moduli of stable maps.  
This will coincide with the perfect obstruction theory used in 
\cite{siebert}.  The functoriality of the cotangent complex would be the 
major advantage of working with derived moduli, as similar statements 
for the perfect obstruction theory are in general difficult to obtain.

An important problem outlined in Gross-Siebert is to identify 
interesting quasi-compact substacks of the derived moduli of stable log 
maps. As the topology of the derived and the underived moduli are the 
same our approach does not suggest anything on this problem.

\begin{appendix}
\section{}\label{sec:appendixA}

\noindent For the readers' convenience, we will give a proof of Proposition \ref{prop:cruciallemma}. 

\begin{proof}[Proof of Proposition \ref{prop:cruciallemma}] Let $\pi\colon R \rightarrow S$ be a square zero extension of discrete commutative rings, and let $J= \ker \pi$ be the corresponding square zero ideal. Then we have to show that there exists a derivation \[d \in \pi_{0}\Map_{R/s\mathcal{A}/S}(S,S \oplus J[1])\] such that there exists an isomorphism in $\mathrm{Ho}(s\mathcal{A}/S)$, between $\pi\colon R \rightarrow  S$ and the canonical projection $p_{d}\colon S \oplus_{d} J \rightarrow S,$ where $p_{d}$ is defined by the homotopy pullback diagram 
$$ \xymatrix@-1pc{S \oplus_{d} J \ar[r] \ar[d]_-{p_{d}} & S \ar[d]^-{0} \\ S \ar[r]_-{d} & S \oplus J[1].}$$

We will give two proofs, one working in any characteristic and the other, considerably simpler, working in characteristic zero. We begin with the general case.

Let $\pi\colon R \rightarrow S$ be a surjection of commutative algebras with square zero ideal $J= \ker \pi$.
As a first step, we apply the functor $-\otimes_{R}^{\mathbb{L}} S$ to  the cofiber sequence \[\xymatrix@-1pc{R \ar[r]^-{\pi} & S \ar[r] & J[1],}\]  and  obtain a split fiber sequence. The splitting map gives a map $$\psi\colon S \otimes_{R}^{\mathbb{L}} S \longrightarrow S \oplus J[1]$$ in $\mathrm{Ho}(S /s\mathcal{A}/ S)$, where \[\xymatrix@-1pc{&  S & \\
S \otimes_{R}^{\mathbb{L}} S \ar[rr] \ar[ur]^-{\mu}  & & S \oplus J[1] \ar[ul]_-{\textrm{pr}_{1}} \\
&  S \ar[ul]^-{j_{1}} \ar[ur]_-{0} & }\] commutes, $\mu$ being induced by the product map, and $j_1$ being induced by $y \longmapsto y\otimes 1$. By computing the action of $\psi$ on homotopy groups, we see that $$\psi_{\leq 1}:= \tau_{\leq 1}(\psi)\colon \tau_{\leq 1}(S \otimes_{R}^{\mathbb{L}} S) \longrightarrow \tau_{\leq 1}(S \oplus J[1]) \simeq S \oplus J[1]$$ is an isomorphism in $\mathrm{Ho}(S /s\mathcal{A}/ S)$.

As a second step, we define $d\colon S \longrightarrow S \oplus J[1]$ as the composite \[\xymatrix{S \ar[r]^-{j_2} & S \otimes_{R}^{\mathbb{L}} S \ar[r] & \tau_{\leq 1}(S \otimes_{R}^{\mathbb{L}} S) \ar[r]^-{\psi_{\leq 1}}  & S \oplus J[1]}\] where $j_2$ is induced by $y\longmapsto 1\otimes y$. Observe that, by the first step, $d$ is a section of the projection $\textrm{pr}_{1}\colon S \oplus J[1] \longrightarrow S$.

As a third step, we observe that since the two composites \[\xymatrix{R \ar[r]^-{\pi} & S \ar[r]^-{j_1} & S \otimes_{R}^{\mathbb{L}} S}, \,\,\,\,\xymatrix{R \ar[r]^-{\pi} & S \ar[r]^-{j_2} & S \otimes_{R}^{\mathbb{L}} S}\] coincide, we get an induced canonical map 
$\alpha\colon R \longrightarrow S\oplus_{d} J$, where $S\oplus_{d} J$ is defined by the homotopy pullback diagram 
\[\xymatrix@-1pc{S\oplus_{d} J \ar[d]_-{p} \ar[r] & S \ar[d]^-{0} \\ S \ar[r]^-{d} & S \oplus J[1].}\] Moreover, if we view $S\oplus_{d} J$ as an object in $\textrm{Ho}(s\mathcal{A}/S)$ via $p$, then $\alpha$ is a morphism in $\textrm{Ho}(s\mathcal{A}/S)$.

By computing the action of $\alpha\colon R \longrightarrow S\oplus_{d} J$ on homotopy groups, it is easy to check that it is an isomorphism in $\textrm{Ho}(s\mathcal{A}/S)$.

We now give an alternative proof in characteristic zero. As 
above, let $\pi\colon R \rightarrow S$ be a surjection of commutative 
algebras with square zero ideal $J= \ker \pi$. If the base 
commutative ring $k$ is a $\mathbb{Q}$-algebra, the homotopy theories of 
simplicial commutative $k$-algebras and of differential non-positively 
graded commutative $k$-algebras (cdga's for short) are equivalent. So we 
are allowed to work with cdga's. Note that $S \oplus J[1]$ can then be 
represented by the cdga \[\xymatrix{0 \ar[r] & J \ar[r]^-{0} & S 
  \ar[r] & 0}\] where $S$ sits in degree $0$. The $0$ derivation is 
then represented by the commutative diagram 
\[\xymatrix@-1pc{0 \ar[r] & 0 \ar[r] \ar[d]^-{0} & S \ar[d]^-{\textrm{id}} 
  \ar[r] & 0 \\ 0 \ar[r] & J \ar[r]^-{0} & S \ar[r] & 0}\]

Observe that we may represent $S$ also by the cdga \[\xymatrix{0 
  \ar[r] & J \ar[r]^-{i} & R \ar[r] & 0,}\] where $i$ denotes the 
inclusion map.  Then we can define a derivation $d$ by the 
commutative diagram \[\xymatrix@-1pc{0 \ar[r] & J \ar[r]^-{i} 
  \ar[d]^-{\textrm{id}} & R \ar[d]^-{\pi} \ar[r] & 0 \\ 0 \ar[r] & J 
  \ar[r]^-{0} & S \ar[r] & 0,}\] and remark that $d$ is a 
fibration of cdga's. Since the model category of cdga's is proper, the 
ordinary pullback of the zero derivation and of $d$ computes the 
homotopy pullback $S\oplus_{d} J$. But the ordinary pullback 
is given by just \[\xymatrix{0 \ar[r] & 0 \ar[r]  & R \ar[r] & 0}\]
(i.e., by just $R$ sitting in degree $0$). So we conclude that there is 
an isomorphism $R \simeq S\oplus_{d} J$ in the homotopy 
category of cdga's$/S$.
\end{proof}
\end{appendix}

%\bibliographystyle{alpha}
%\bibliography{dlog}

\end{document}